\documentclass[final,hidelinks,11pt,reqno,nopagebackref]{amsart}
\usepackage{amsmath,amssymb,amsfonts}
\usepackage[mathscr]{eucal}
\usepackage{amsthm}
\usepackage[arrow]{xy}
\usepackage{stmaryrd}
\usepackage[margin=1.06 in]{geometry}
\usepackage[colorinlistoftodos]{todonotes}
\usepackage[normalem]{ulem}

\usepackage[US,bibtex,notcite, notref]{mathieu}
\allowdisplaybreaks[1]
\usepackage{verbatim}

\theoremstyle{definition}

\DeclareMathOperator{\pbw}{pbw}
\DeclareMathOperator{\MC}{MC}

\newcommand{\ci}{C^\infty}
\newcommand{\ff}{\mathfrak}
\newcommand{\Cc}[1]{\mathcal{#1}}
\newcommand{\p}{\partial}
\renewcommand{\L}{L_\infty}
\newcommand{\F}[2]{\frac{#1}{#2}}

\newcommand{\Ber}{\mathrm{Dens}}
\newcommand{\Berr}{\mathrm{Ber}}
\renewcommand{\ker}{\mathrm{ker}}
\newcommand{\im}{\mathrm{im}}

\newcommand{\sden}{\sqrt{\Cc{D}x}}

\newcommand{\bs}{\textbf{s}}
\newcommand{\bt}{\textbf{t}}
\newcommand{\bX}{\textbf{X}}
\newcommand{\bY}{\textbf{Y}}
\newcommand{\bZ}{\textbf{Z}}
\newcommand{\bS}{\textbf{S}}

\newcommand{\brho}{\boldsymbol\rho}

\newcommand{\BV}{\mathrm{BV}}
\newcommand{\DO}{\mathrm{DO}}
\newcommand{\B}{\mathrm{Dens}^{1/2}_{\Cc{M}}}
\newcommand{\dDO}{\DO(\B)}
\newcommand{\Ba}[1]{\mathrm{Dens}^{1/2}_{#1}}
\newcommand{\Gr}{\mathrm{Gr}}
\newcommand{\Sh}{\mathrm{Sh}}

\newcommand{\Op}{\mathrm{Op}}

\newcommand{\allbracket}{\big\{}
\newcommand{\arrbracket}{\big\}}
\newcommand{\ZZ}{\mathbb{Z}}
\newcommand{\zz}{\mathbb{Z}}
\newcommand{\NN}{\mathbb{N}}

\newcommand{\RR}{\mathbb{R}}
\newcommand{\hbarr}{\llbracket\hbar\rrbracket}
\newcommand{\frakg}{{\mathfrak g}}

\newcommand{\cM}{\mathcal{M}}
\newcommand{\cA}{\mathcal{A}}
\newcommand{\cF}{\mathcal{F}}
\newcommand{\cinf}[1]{C^{\infty}(#1)}
\newcommand{\sections}[1]{\Gamma\big(#1\big)}
\newcommand{\aV}{A}
\newcommand{\XX}{\mathfrak{X}}
\newcommand{\half}{\frac{1}{2}}
\newcommand{\ih}{\frac{1}{\hbar}}
\newcommand{\hbarQ}{ \hbar \Cc{L}_Q}
\newcommand{\xto}[1]{\stackrel{{#1}}{\longrightarrow}}
\newcommand{\hdo}[2]{\frac{\hbar\DO^+_{#1} (\B)}{\hbar\DO^+_{#2} (\B)}}
\newcommand{\prr}{\phi}
\newcommand{\aDelta}{X}
\newcommand{\formal}{ formal power }

\title{Quantization of  (-1)-Shifted Derived Poisson Manifolds}

\author{Kai ~Behrend,  ~Matt ~Peddie and ~Ping ~Xu\\
\\
Dedicated to Jean-Luc Brylinski on his 70th birthday}

\newcommand{\Addresses}{{
  \bigskip
  \footnotesize

  Kai Behrend, \textsc{Department of Mathematics,  University of British Columbia}\par\nopagebreak
  \textit{E-mail address}: \texttt{behrend@math.ubc.ca}

  \medskip

 Matt Peddie, \textsc{Department of Mathematics, Pennsylvania State University}\par\nopagebreak
  \textit{E-mail address}: \texttt{matt.peddie11@gmail.com}

  \medskip

  Ping Xu, \textsc{Department of Mathematics, Pennsylvania State University}\par\nopagebreak
  \textit{E-mail address}: \texttt{ping@math.psu.edu}

}}

\begin{document}

\begin{abstract}
We investigate the quantization problem of $(-1)$-shifted derived Poisson manifolds
in terms of $\BV_\infty$-operators on the space of Berezinian half-densities.
We prove that quantizing  such a $(-1)$-shifted derived Poisson manifold is
equivalent to the lifting of a  consecutive sequence 
of Maurer-Cartan elements, each obtained from a short exact sequence of differential
graded Lie algebras. At each step,  the obstruction
 is  a certain class in the second Poisson cohomology. Consequently,
a $(-1)$-shifted derived Poisson manifold is quantizable if the
second Poisson cohomology group vanishes.
We also prove that for any  $\L$-algebroid $\Cc{\aV}$,
its corresponding  linear $(-1)$-shifted   derived  Poisson manifold
$\Cc{\aV}^\vee[-1]$  admits a canonical quantization.
Finally, given a Lie algebroid $A$ and a   one-cocycle  $s\in \sections{A^\vee}$,
the $(-1)$-shifted derived Poisson manifold corresponding to
the derived intersection of coisotropic
submanifolds determined by the graph of $s$ and the zero section of
the Lie-Poisson $A^\vee$ is shown to
admit a canonical quantization in terms of Evens-Lu-Weinstein module. 
\end{abstract}

\maketitle

\section*{Introduction}

The notion of homotopy Schouten algebras was introduced by
Khudaverdian-Voronov in their seminar paper  \cite{MR2757715} 
in 2008. In  particular,  they discovered  a homotopy
analogue of  Koszul-Brylinski construction
 \cite{MR950556, MR837203}
 that the space of differential forms on a $P_\infty$-manifold
 \cite{MR2304327, MR2180451}
admits a canonical homotopy Schouten algebra structure.
Khudaverdian-Voronov's construction is mainly in the context
of supergeometry, i.e. $\ZZ_2$-grading.
In the context of $\ZZ$-grading, following 
Pridham~\cite{MR3653066,arXiv:1804.07622},
 the underlying geometric object in the  $C^\infty$-context
 is called
  $(-1)$-shifted derived Poisson manifolds \cite{MR4091493}.

A  $(-1)$-shifted derived Poisson manifold is a $\mathbb{Z}$-graded manifold
 $\Cc{M}$ whose algebra of functions $\ci(\Cc{M})$ is a
homotopy Schouten algebra, called   $(+1)$-shifted derived Poisson algebra
in the paper. Equivalently, a  $(-1)$-shifted derived Poisson
 manifold $\Cc{M}$ is a   differential graded (dg) manifold $(\Cc{M},Q)$ 
equipped with a degree $(+1)$ \formal  series $\Pi\in\Gamma(\hat{S}T_{\Cc{M}})$,
 where $\Pi = \sum_{n\geq 2}\Pi_n$ with $\Pi_n \in
\Gamma({S}^n T_{\Cc{M}})$,  satisfies the Maurer-Cartan equation
  \begin{equation}
        \allbracket Q,\Pi\arrbracket + \F{1}{2}\allbracket\Pi,\Pi\arrbracket = 0.
        \end{equation}
Here the bracket $ \allbracket- , -\arrbracket$
refers to the standard Poisson bracket on $\Gamma(\hat{S}T_{\Cc{M}})
\cong C^\infty (T^\vee_\cM)$ corresponding to
the canonical symplectic structure on $T^\vee_\cM$. 
We  often  write $\Pi_1 =Q$, and call $Q+\Pi$  the Poisson tensor.
A  $(-1)$-shifted derived Poisson manifold is usually denoted
$(\cM, Q, \Pi)$.

This  paper is devoted  to  the study of   quantizations  of 
$(-1)$-shifted derived Poisson manifolds.    Following  Kravchenko
 \cite{MR1764440}, a quantization of a $(-1)$-shifted derived Poisson manifold
$(\cM, Q, \Pi)$ is a square-zero  differential operator
of degree $(+1)$,  defining 
a   Batalin-Vilkovisky algebra up to homotopy,
 whose associated
$(+1)$-shifted derived Poisson algebra is the  one on $C^\infty (\cM)$
corresponding to $(\cM, Q, \Pi)$.
For quantizations being  intrinsic, instead of differential operators
on $C^\infty (\cM)$, we consider differential operators
on  $\Gamma(\B)$,  the space of Berezinian half-densities on $\Cc{M}$.
More precisely, a  quantization of a $(-1)$-shifted derived Poisson manifold
 $(\Cc{M},Q,\Pi)$ is a $\BV_\infty$-operator $\Delta \in\hbar\DO^+ (\B) $,
i.e. a  square-zero  self-adjoint $\hbar$-enhanced differential operator,
whose     extended principal symbol, evaluated at
$\hbar=1$,  equals to $Q + \Pi$. See Definition
\ref{Section 4 Defn Quantisation of Poisson structure}.
In \cite{MR3685170}, Khudaverdian-Peddie constructed 
an example of non-quantizable $(-1)$-shifted derived Poisson manifold,
which contains only one term $\Pi_2\in 
\Gamma({S}^2 T_{\Cc{M}})$ with $Q$ being zero, i.e. $(-1)$-shifted
Poisson manifold. See Example \ref{ex:Hovik}.

Our first main result is the following   theorem describing the 
obstruction class to the quantizations.

 \begin{introthm}[Theorem~\ref{thm:main}]
Let  $(\Cc{M},Q,\Pi)$ be a $(-1)$-shifted derived Poisson manifold.
Assume that the second Poisson cohomology 
group $\Cc{H}^2 (\Cc{M}, Q+\Pi)$ vanishes, then $(\Cc{M},Q,\Pi)$
is quantizable.
\end{introthm}

Our strategy of proof is to convert the quantization problem
into the problem of lifting Maurer-Cartan elements of a  short
exact sequence of differential graded Lie algebras (dglas).
For a given dg manifold $(\Cc{M},Q)$, a $(-1)$-shifted
 derived Poisson
structure $\Pi$ is equivalent  to a Maurer-Cartan element
of the dgla
$\big(\Gamma(\hat{S}T_\Cc{M}), \quad \{\cdot, \cdot\},  \quad\{Q, \cdot\}\big)$, while 
 a $\BV_\infty$-operator
 $\Delta\in\hbar\DO^+ (\B)$, or more precisely
 $\Delta-\hbarQ$,
 corresponds to   a Maurer-Cartan element
of the  dgla
$\big( \hbar\DO^+ (\B) ,  \quad [\cdot, \cdot]_\hbar,  \quad[\hbarQ, \cdot]_\hbar\big)$.
These two dglas are related by
the  short exact sequence
\begin{equation}
\label{eq:dgla-exact1}
0\xto{}\hbar\DO^+_2 (\B)\xto{i}\hbar\DO^+ (\B)\xto{\phi}\Gamma(\hat{S}T_\Cc{M})\xto{}0,
\end{equation}
where $\hbar\DO^+_2 (\B)$ is considered as a dg Lie subalgebra
of $\hbar\DO^+ (\B)$, the morphism  $i$ is the inclusion map and
$\phi$ is the morphism defined by
the  extended principal symbol, evaluated at
$\hbar=1$.
%
%
Here,  for any integer $t\geq 0$, $\hbar\DO^+_t (\B)$, denotes the
space of those $\hbar$-enhanced differential operators $\Delta$
such that $\F{1}{\hbar^t} \Delta \in \hbar\DO^+ (\B)$.  
Therefore, the lifting of a Maurer-Cartan element
in the short exact sequence \eqref{eq:dgla-exact1}
is equivalent to the lifting of a  consecutive sequences
of Maurer-Cartan elements, each obtained from
  the  short exact sequences of dglas:
\begin{equation}
\label{eq:dgla-exact-k1}
0\xto{}\hdo{2k}{2k+2} \xto{i_k} \hdo{}{2k+2} \xto{\prr_k} \hdo{}{2k}\xto{}0,
\end{equation}
 $k=1, 2, \cdots$.
Here $i_k$ is the natural inclusion and
  $\prr_k$ is the natural projection.
These are indeed square-zero extensions of dglas. The obstruction
to lifting a Maurer-Cartan element is  proved to be a certain
class in the second Poisson cohomology  
$\Cc{H}^2 (\Cc{M}, Q+\Pi)$.

Theorem {\bf A} describes the obstruction to
the existence of quantizations, 
 which is a sufficient condition for
quantizing $(-1)$-shifted derived Poisson manifolds.
In many situation, however,  although the second
Poisson cohomology group $\Cc{H}^2 (\Cc{M}, Q+ \Pi)$
 does not vanish,
 $(\Cc{M},Q,\Pi)$ may still be quantizable.

An important  class of $(-1)$-shifted derived Poisson manifolds
arises as \emph{linear} Poisson structures on a vector bundle, which
corresponds to $\L$-algebroids, similar
 to the ordinary Lie-Poisson construction \cite{MR1269545, MR2178041}.
In other words,  $\L$-algebroid structures
 on  a  $\ZZ$-graded vector bundle $\Cc{\aV}$ are
in one-one correspondence  to  linear
 $(-1)$-shifted derived Poisson manifold structures on $\Cc{\aV}^\vee[-1]$
\cite{MR4091493}. 
Our second main theorem is  the following

\begin{introthm}[Theorem~\ref{thm:main2}]
For any  $\L$-algebroid $\Cc{\aV}$,
its corresponding  linear $(-1)$-shifted   derived  Poisson manifold
$\Cc{\aV}^\vee[-1]$  admits a canonical quantization.
\end{introthm}

Our  construction of the $\BV_\infty$-operator utilizes
the fiberwise $\hbar$-Fourier transform  \cite{MR4235776}
 by Voronov-Zorich~\cite{MR947602}:
\begin{equation}
\label{eq:Fourier0}
\cF: \sections{\Ba{\Cc{\aV}[1]}}\xto{\simeq} \sections{ \Ba{\Cc{\aV}^\vee[-1]}}.
\end{equation}
To any  $\L$-algebroid $\Cc{\aV}$, there  is
a Chevalley--Eilenberg differential $D$,
which is a homological vector field on $\Cc{\aV}[1]$.
Therefore, $(\Cc{\aV}[1], D)$ is  a dg manifold.
We prove that $\Delta:= \cF \circ \Cc{L}_{\hbar D} \circ \cF^{-1}
\in \hbar\DO^+( {\mathrm{Dens}^{1/2}_{ \Cc{\aV}^\vee[-1] }})$ 
is a $\BV_\infty$-operator quantizing the linear
$(-1)$-shifted   derived  Poisson manifold on $\Cc{\aV}^\vee[-1]$.
If $\cM$ is a  $P_\infty$-manifold
 \cite{MR2304327, MR2180451}, its cotangent bundle 
$T^\vee_\cM$  naturally carries  an $\L$-algebroid structure \cite{MR2757715},
 Thus $T_\cM [-1]$ is
a linear $(-1)$-shifted   derived Lie-Poisson manifold, whose algebra of functions
are differential forms on $\cM$.
Its quantization has been   studied  by  
Khudaverdian-Voronov  \cite{MR2757715} and more 
 recently  by Shemyakova \cite{MR4235776}.

Another important class of  
$(-1)$-shifted derived Poisson manifolds are derived intersections of
coisotropic submanifolds in a Poisson manifold. As a toy model,
we consider the  Lie-Poisson manifold corresponding to a Lie algebroid.
Let $A$ be a Lie algebroid and $s\in \Gamma (A^\vee)$
a Lie algebroid 1-cocycle, i.e.
a smooth section satisfying the condition that
$d_{\text{CE}}s=0$, where $d_{\text{CE}}$ is the
Chevalley--Eilenberg differential of the Lie algebroid
$A$. Both the graph of $s$ and the zero section  are  coisotropic submanifolds
of the Lie-Poisson manifold $A^\vee$. Their derived  intersection
is a  $(-1)$-shifted derived Poisson manifold
 $(A^\vee [-1], \iota_s, \Pi_2)$,
where $\Pi_2\in  \sections{S^2 T_{ A^\vee [-1]}}$ is the Poisson tensor
determined by the Schouten algebra on $C^\infty ( A^\vee [-1])\cong
 \sections{\Lambda^{-\bullet}A}$.
In this  situation, the  $\BV_\infty$-operator is related to
the  Evens-Lu-Weinstein module  \cite{MR1726784}.

\begin{introthm}[Theorem~\ref{thm:main3}]
Let $A$ be a Lie algebroid and $s\in \Gamma (A^\vee)$
a Lie algebroid 1-cocycle.
Assume that $M$ is an orientable manifold and the  vector bundle $A\to M$
 is also orientable as well.
%
 Then the $(-1)$-shifted derived Poisson manifold
$(A^\vee [-1], \iota_s, \Pi_2)$
admits a canonical quantization:
\begin{equation}
\Delta=\hbar \iota_s+ \hbar^2 \Phi \circ d_{\text{CE}}^{\text{ELW}} \circ \Phi^{-1}
: \sections{\Lambda^{-\bullet } A \otimes \big( \Lambda^{\text{top}} A^\vee \otimes
\Lambda^{\text{top}}T^\vee_M\big)^\half }
\to \sections{ \Lambda^{ -(\bullet+1)} A \otimes \big( \Lambda^{\text{top}} A^\vee \otimes
\Lambda^{\text{top}}T^\vee_M\big)^\half }.
\end{equation}
Here $\Phi$ is the canonical isomorphism \cite[Section 5]{AX03} \cite[Section 6.3]{MR4312284}:
\begin{equation}
\Phi:\sections{ \Lambda^k A^\vee \otimes \big( \Lambda^{\text{top}} A \otimes
\Lambda^{\text{top}}T^\vee_M\big)^\half }
\xto{\simeq}
\sections{ \Lambda^{ \text{top}-k} A \otimes \big( \Lambda^{\text{top}} A^\vee \otimes
\Lambda^{\text{top}}T^\vee_M\big)^\half },
\end{equation}
\end{introthm}
for $k=0, 1, \cdots$.

In particular, when $A$ is $T_M$ and $s=df\in \Omega^1 (M)$ an exact
one-form, where $f\in C^\infty (M)$,  the line bundle
$\big( \Lambda^{\text{top}} A \otimes \Lambda^{\text{top}}T^\vee_M\big)^\half $ is canonically isomorphic to the trivial line bundle $M\times \RR$
and
$d_{\text{CE}}^{\text{ELW}}$ reduces to the ordinary de Rham differential.
As an immediate consequence, we see that
\begin{equation}
\hbar \iota_{df}+ \hbar^2 \Phi \circ d_{DR} \circ \Phi^{-1}
: \sections{\Lambda^{-\bullet } T_M \otimes \big( \Lambda^{\text{top}}
T^\vee_M\big) }
\to
\sections{\Lambda^{-(\bullet+1) } T_M \otimes \big( \Lambda^{\text{top}}
T^\vee_M\big) }
\end{equation}
is  a
$\BV_\infty$-operator quantizing  the
$(-1)$-shifted derived symplectic manifold
 $(T_M^\vee [-1], \iota_{df}, \omega_{\text{can}})$,
where $\Phi: \Omega^k (M)\to \sections{
\Lambda^{ \text{top}-k} T_M \otimes \big( \Lambda^{\text{top}}T^\vee_M\big)
}, \forall k=0, 1, \cdots$ is the canonical isomorphism.

Finally, we note that several works on related
subject
have appeared recently   in the literature. For instance,
 we refer the readers to
\cite{MR3654355, MR3235797, MR3782422, MR3685170, MR1952122,  MR2070054, arXiv:1804.07622, MR4235776, MR4009180} and references therein.
In particular, we note that recently
Bandiera proved the homotopy transfer theorem for $\BV_\infty$-algebras
and showed that the homotopy transfer is compatible with
quantizations \cite{https://doi.org/10.48550/arxiv.2012.14812}.
The work in this paper is presented  in the context of  $\ZZ$-graded manifolds.
However,  it works for supermanifolds (i.e. $\ZZ_2$-graded) as well.

\color{black}

\subsection*{Notation}
Let $V=\oplus_{n\in\mathbb{Z}} V_n$ be a $\mathbb{Z}$-graded vector space over
$\mathbb{R}$,
  where an element $v\in V_n$ has degree $n$, written $|v| = n$. For $k\in\mathbb{Z}$, the symbol $V[k]$ denotes the $k$-shifted space such that $(V[k])_n:= V_{n+k}$. In particular, an element $v\in V$ of degree $n$ lies in $(V[k])_{n-k}$ if viewed as an element of $V[k]$. The dual vector space $V^\vee$ is graded according to $(V^\vee)_n:=(V_{-n})^\vee$, which ensures the non-degenerate paring $V\times V^\vee\rightarrow\mathbb{R}$ carries degree zero.

Given a graded vector space $V$, let $S(V)$ denote the symmetric algebra of $V$, where $S(V) = \oplus_{m\geq 0}S^m(V)$ decomposes into homogeneous terms graded naturally by weight. A weight $m$, degree $l$ element of $S(V)$ is thus
	\begin{equation*}
	v_{1}\odot\cdots\odot v_{m}\in S^m(V),\qquad |v_1|+\cdots+|v_m|=l.
	\end{equation*}
The symbol $\hat{S}(V)=\prod^\infty_{m=0}S^m(V)$ denotes the $\ff{m}$-adic completion of $S(V)$ with respect to the ideal $\ff{m}$ generated by $V$.

A $\mathbb{Z}$-graded manifold $\Cc{M}$ is a smooth manifold $M$ (called the support), together with a sheaf of $\mathbb{Z}$-graded commutative
 $C^\infty_{M}$-algebras over $M$, isomorphic to $\ci(U)\otimes_\RR \hat{S}(V^\vee)$ for
 any  sufficiently small open  neighborhood $U\subset M$,
where $V$ is a fixed
$\mathbb{Z}$-graded vector space, and
$C^\infty_M$ denotes  the sheaf
of $\RR$-valued smooth functions over $M$.
 If $M$ and $V$ are both finite dimensional,
 then $\Cc{M}$ is said to be of finite dimension. Throughout this work,
 we only  deal  with finite dimensional $\mathbb{Z}$-graded manifolds.
A dg manifold is a $\mathbb{Z}$-graded manifold  equipped with a
homological vector field, i.e. a vector field $Q$  of degree $(+1)$ such
that $Q^2=0$.


\section{$(-1)$-Shifted Derived Poisson Manifolds}

Our definitions and conventions follow those in \cite{MR4091493}, 
adopted from \cite{MR3653066, arXiv:1804.07622}.

\subsection{Definition}

\begin{definition}
\label{def:derived}
A   $(+1)$-shifted derived Poisson algebra is a $\mathbb{Z}$-graded
commutative algebra $\Cc{A}$ endowed with a sequence of degree
 $(+1)$ multi-linear maps $\lambda_n:\Cc{A}^{\otimes n}\rightarrow \Cc{A}$,
 $n\geq 1$, called {\em Poisson multi-brackets},
 defining a  $\L[1]$-algebra structure on $\Cc{A}$, and such that
for any all $n\geq 1$, and homogeneous  elements $a_1,\ldots,a_{n-1}\in\Cc{A}$
the map
	\begin{equation*}
 	a \mapsto \lambda_n(a_1,\ldots,a_{n-1},a)
	\end{equation*}
is a graded derivation  of degree $(1+|a_1|+\cdots |a_n|)$.
\end{definition}

\begin{remark}
In general, for any
 $k\in\mathbb{Z}$, 
 one can consider $k$-shifted derived Poisson algebras
  \cite{MR4091493}.
 For $k=0$, such algebras are known as homotopy Poisson algebras \cite{MR2757715}, or $P_\infty$-algebras \cite{MR2304327, MR2180451}.
 The associated
 $\mathbb{Z}_2$-graded derived Poisson algebras are known as homotopy Poisson and homotopy Schouten (or $S_\infty$) algebras for even and odd $k$ respectively,
which are due to  Khudaverdian-Voronov 
\cite{MR2757715, MR2163405}.
 The general behavior of these derived Poisson structures depends on the parity of the shift, and not on the value of the integer itself.
\end{remark}

\begin{definition}
A morphism $\phi:\Cc{A}\rightsquigarrow \Cc{B}$ of  $(+1)$-shifted derived
 Poisson algebras is a sequence  $(\phi_i)_{i\geq 1}$ of degree zero
 linear maps
 $\phi_i:S^i\Cc{A}\rightarrow \Cc{B}$, which defines
 an $\L[1]$-morphism from $\Cc{A}$ to $\Cc{B}$,
 and in addition  satisfies the derivation property:
    \begin{equation*}
    \phi_{n+1}(a_1,\ldots,a_{n},bc) = \sum^n_{l=0}\sum_{\tau\in \Sh(l,n-l)}\varepsilon (\tau)\phi_{l+1}(a_{\tau(1)},\ldots,a_{\tau(l)},b)
\phi_{n-l+1}(a_{\tau(l+1)},\ldots,a_{\tau(n)},c),
    \end{equation*}
where $a_1,\ldots,a_n,b,c\in\Cc{A}$ are homogeneous elements, $\Sh(l,n-l)$
 denotes the $(l,n-l)$-shuffles, and $\varepsilon(\tau)$ is  the Koszul
 sign.
\end{definition}

\begin{definition}
A  $(-1)$-shifted derived Poisson manifold is a $\mathbb{Z}$-graded manifold $\Cc{M}$ whose algebra of functions $\ci(\Cc{M})$ is a  $(+1)$-shifted derived Poisson algebra.
\end{definition}

Similar to an ordinary  Poisson manifold which is defined by an anti-symmetric
 degree $2$ bivector field, a degree $(-1)$-shifted derived Poisson manifold
 structure is equivalent to a \formal series of symmetric contravariant tensor fields on $\Cc{M}$ of degree $(+1)$.
A degree $l$ vector field $X\in \Gamma(T_\Cc{M})$ as a section of the tangent bundle $T_{\Cc{M}}$ is a derivation $X:\ci(\Cc{M})\rightarrow\ci(\Cc{M})$ of degree $l$. The completed symmetric algebra
    \begin{equation*}
     \Gamma(\hat{S}T_\Cc{M}): = \prod_{n\geq 0}\Gamma(S^nT_\Cc{M})
    \end{equation*}
then consists of \formal series of symmetric contravariant tensor fields. It comes furnished with the degree $0$ Poisson bracket:
    \begin{equation}\label{eq:Manchester}
    \allbracket X_1\odot\cdots\odot X_n, \ Y_1\odot\cdots\odot Y_m\arrbracket = \sum^n_{k=1}\sum^m_{l=1}\delta_k\varepsilon_l \bX^{\{k\}}\odot[X_k,Y_l]\odot \bY^{\{l\}},
    \end{equation}
where $\bX^{\{k\}} = X_1\odot\cdots\odot X_{k-1}\odot X_{k+1}\odot\cdots\odot X_n$, $\bY^{\{l\}} = Y_1\odot\cdots\odot Y_{l-1}\odot Y_{l+1}\odot\cdots\odot Y_n$, and the graded signs are given by the Koszul rule
 $\delta_k = (-1)^{|X_k|(|X_{k+1}|+\cdots+|X_n|)}$ and $\varepsilon_l=(-1)^{|Y_l|(|Y_{1}|+\cdots+|Y_{l-1}|)}$. 


\begin{remark}
The completed symmetric algebra $\Gamma(\hat{S}T_\Cc{M})$ may be naturally
 identified with the algebra of fiberwise
\formal  series on $T^\vee_\Cc{M}$, whence the degree 0 Poisson bracket
\eqref{eq:Manchester} coincides with the degree 0 Poisson bracket
 arising from the canonical symplectic structure on $T^\vee_\Cc{M}$.
\end{remark}

\begin{theorem}\cite{MR3653066, arXiv:1804.07622, MR4091493}
\label{Section 1 Thm Shifted Poisson structure equiv to formal series}
A  $(-1)$-shifted derived Poisson manifold $\Cc{M}$ is equivalent to 
a  differential graded (dg) manifold $(\Cc{M},Q)$ equipped with a degree 
$(+1)$ \formal series $\Pi\in\Gamma(\hat{S}T_{\Cc{M}})$, 
where $\Pi = \sum_{n\geq 2}\Pi_n$ with $\Pi_n\in \sections{S^n T_\cM}$,
 satisfies  the Maurer-Cartan equation
	\begin{equation}
\label{Section 1 Eqn Maurer-Cartan equation for poisson structure}
	\allbracket Q,\Pi\arrbracket + \F{1}{2}\allbracket\Pi,\Pi\arrbracket = 0.
	\end{equation}
Here $\allbracket -, -\arrbracket$ stands 
for the canonical Poisson bracket \eqref{eq:Manchester}.
\end{theorem}
As in  classical Poisson geometry,  $Q+\Pi$ is often
 called a  \emph{Poisson tensor}.

Let $\cM$ and $\cM'$ be $(-1)$-shifted derived Poisson manifolds
 with  Poisson multi-brackets
$\lambda_l: $ $(\cinf{\cM})^{\otimes l}$ $\to$ $\cinf{\cM}$
and $\lambda'_l: $ $(\cinf{\cM'})^{\otimes l}$ $\to$ $\cinf{\cM'}$, $ l\geqslant 1$, respectively. 
A morphism of $(-1)$-shifted derived Poisson manifolds from $\cM$
 to $\cM'$ is a map of $\mathbb{Z}$-graded manifolds
 $\phi: \cM\to \cM'$ together with a collection of maps
\[ \varphi_n: (\cinf{\cM'})^{\otimes n}\rightarrow \cinf{\cM} ,\quad n=2,3,\cdots \]
such that $\varphi_\infty$ $=$ $(\varphi_1=\phi^*, \varphi_2,\varphi_3,\cdots)$ is a morphism of degree $(+1)$ derived Poisson algebras
from $(\cinf{\cM'},$ $\lambda'_1$, $\lambda'_2$, $\cdots$ , $\lambda'_n$, $\cdots)$
to $(\cinf{\cM},\lambda_1, \lambda_2, \cdots , \lambda_n, \cdots)$
 \cite{MR4091493}.
In particular, $\phi: \cM\to\cM'$ is a map of dg manifolds.
Geometrically, these morphisms $\varphi_\infty$   correspond to
the thick morphisms of graded manifolds due to
  Voronov \cite{MR3588925, MR3894641, MR4035066,  arxiv.1808.10049}.


\begin{example}\label{Section 1 Ex Lie Poisson brackets}
For a finite dimensional $\L$-algebra $\frakg$, the  multi-brackets
$l_n:S^n (\frakg [1]) \rightarrow \frakg [1] $
on the $\L[1]$-algebra $\frakg[1]$ may be extended via the Leibniz rule
 to equip $\frakg^\vee[-1]\cong (\frakg [1])^\vee$
  with the structure of a  $(-1)$-shifted derived Poisson manifold.
 Define
    \begin{equation*}
    \lambda_n(\xi_{a_1},\ldots,\xi_{a_n}) := \pm C^b_{a_1\ldots a_n}\xi_b,
    \end{equation*}
where $\xi_{a_1},\ldots,\xi_{a_n}$ are the
coordinate functions on $\frakg^\vee[-1]$ corresponding to a
 basis $\{e_a\}$ of $\frakg[1]$,
 and $C^b_{a_1\ldots a_n}$ are the degree $(+1)$
 structure constants of the $n$th multi-bracket $l_n$.
\end{example}

\begin{example}\label{example-three}
Let $A$ be a Lie algebroid and $s\in \Gamma (A^\vee)$ a Lie
algebroid 1-cocycle, i.e.
a smooth section satisfying the condition that
$d_{\text{CE}}s=0$, where $d_{\text{CE}}$ is the 
Chevalley--Eilenberg differential of the Lie algebroid
$A$. Consider the dg manifold $(A^\vee [-1], Q)$,
where the algebra of functions $C^\infty ( A^\vee [-1]) \cong
 \sections{\Lambda^{-\bullet}A}$ 
and $Q=\iota_s$, the interior product with $s$. This dg manifold
describes the derived  intersection of
the graph of $s$ with the zero section of  $A^\vee$.
Note that both the graph of $s$ and the zero section
are coisotropic submanifolds of the Lie-Poisson
manifold $A^\vee$.
Let $\Pi_2\in  \sections{S^2 (T_{A^\vee [-1]})}$ be the Poisson
tensor defining the Schouten bracket on
$C^\infty ( A^\vee [-1])\cong
\sections{\Lambda^{-\bullet}A}$. Then $(A^\vee [-1], \iota_s, \Pi_2)$
is indeed a $(-1)$-shifted derived Poisson manifold, which
describes the derived  intersection of coisotropic submanifolds
$\text{graph}(s)$ and the zero section
of  the Lie-Poisson
manifold $A^\vee$.

In particular, when $A$ is $T_M$ and $s\in \Omega^1 (M)$ is a closed
one-form, we have a $(-1)$-shifted derived Poisson manifold 
$(T^\vee_M[-1], \iota_s, \Pi_2)$, which
is indeed a $(-1)$-shifted derived symplectic manifold, denoted
$(T_M^\vee [-1], \iota_{s}, \omega_{\text{can}})$. It
is the  derived  intersection, in the
symplectic manifold $T^\vee_M$,
 of two Lagrangian submanifolds: the graph of $s$ and the zero section.
\end{example}

The following example is due to Khudaverdian-Voronov \cite{arxiv.1808.10049}. See also \cite{MR4235776}.

\begin{example}\label{Section 1 Ex Koszul brackets}
A homotopy Poisson structure on a  $\zz$-graded manifold $\Cc{M}$, according to Khudaverdian-Voronov
 \cite{MR2757715},
is defined by a degree $(+2)$ \formal series
 $P = \sum_{n\geq 1}P_n\in\Gamma(\hat{S}(T_\Cc{M}[-1]))$,
 satisfying  the equation $\llbracket P,P\rrbracket=0$,
 where the canonical Schouten bracket
$\llbracket \cdot , \cdot \rrbracket$ carries degree $(-1)$\footnote{In this
case, $C^\infty (\cM)$ is equipped with a $P_\infty$-algebra structure \cite{MR2304327, MR2180451}.}.
 Thus $P$ corresponds to
 a degree $(+2)$ \formal series in $\ci(T^\vee_\Cc{M}[1])$, whose
  Hamiltonian  vector field $d_P$ is a
homological vector field  on $T^\vee_\Cc{M}[1]$.
It in turn  determines a degree $(+1)$ fiberwise linear function
 on the cotangent bundle $h_P\in\ci(T^\vee_{T^\vee_{\Cc{M}}[1]})$ such that
    \begin{equation*}
    \{h_P,-\} = d_P,\qquad \{h_P,h_P\}=0,
    \end{equation*}
for the canonical degree $0$ Poisson bracket on $T^\vee_{T^\vee_{\Cc{M}}[1]}$.
 Via the canonical isomorphism of double vector bundles
 $T^\vee_{T^\vee_\Cc{M}[1]}\cong T^\vee_{T_\Cc{M}[-1]}$,
one obtains a degree $(+1)$ function on $T^\vee_{T_\Cc{M}[-1]}$
 satisfying the Maurer-Cartan equation, and hence a
 $(-1)$-shifted derived Poisson structure on $T_\Cc{M}[-1]$.
\end{example}

\begin{remark}
In the $\mathbb{Z}_2$-grading case, the multi-brackets
corresponding to the $(-1)$-shifted derived Poisson structure on
 $T_\Cc{M}[-1]$ in Example
\ref{Section 1 Ex Koszul brackets} are called the higher
Koszul brackets \cite{MR2757715, arxiv.1808.10049, MR4235776}. 
In this   case, if there is only a  single
binary bracket, it reduces to
 the classical Koszul bracket \cite{MR837203, MR950556, MR1675117}
 defined on differential forms.
\end{remark}

\subsection{Poisson cohomology}

Let $(\Cc{M}, Q, \Pi)$ be a $(-1)$-shifted derived Poisson manifold. The Lichnerowicz differential $d_\Pi:\Gamma(\hat{S}T_{\Cc{M}})\rightarrow \Gamma(\hat{S}T_{\Cc{M}})$ of degree $(+1)$ is defined by
	\begin{equation*}
	d_\Pi =\Cc{L}_Q+ \allbracket \Pi, \ \ -\arrbracket=
  \allbracket Q+\Pi, \ \ -\arrbracket,\qquad |d_\Pi| = 1.
	\end{equation*}
That $d^2_\Pi = 0$ is equivalent to the Maurer-Cartan equation
 \eqref{Section 1 Eqn Maurer-Cartan equation for poisson structure}.
 Then  $(\Gamma(\hat{S}T_\Cc{M}),d_\Pi)$ is a cochain complex,
whose cohomology groups form the Poisson  cohomology
of the $(-1)$-shifted derived Poisson manifold $\Cc{M}$:
    \begin{equation*}
\Cc{H}^k(\Cc{M}, Q+ \Pi):= \F{\ker(d_\Pi:\Gamma(\hat{S}T_\Cc{M})_k\rightarrow \Gamma(\hat{S}T_\Cc{M})_{k+1})}{\im(d_\Pi:\Gamma(\hat{S}T_\Cc{M})_{k-1}\rightarrow \Gamma(\hat{S}T_\Cc{M})_k)},\qquad k\in \ZZ.
	\end{equation*}

\begin{example}
    Let $\Cc{M}$ be a $(-1)$-shifted derived
 Poisson manifold with a single binary bracket defined by the degree $(+1)$ symmetric tensor field $\Pi_2$.
Assume  that $\Pi_2$ is nondegenerate.
Hence $\Cc{M}$ is a $(-1)$-shifted symplectic manifold, and
 there exists a  bundle isomorphism $T^\vee_\Cc{M}\stackrel{\simeq}{\to}
 T_\Cc{M}[1]$ defined by $\alpha\mapsto \Pi_2(\alpha,-)$, for
any differential one form $\alpha\in\Gamma(T^\vee_\Cc{M})$.
    
    This isomorphism induces an isomorphism between Poisson and de Rham cohomologies in the usual sense.
\end{example}


\section{$BV_\infty$-operators on Berezinian half-densities}

\subsection{Differential operators on half-densities}
\label{sec:2.1}

In this section,
 we recall some basic properties regarding differential operators acting on
 half densities \cite{MR1202882, MR4235776,  MR3307151, MR1701597}.

Here for a $\zz$-graded manifold $\Cc{M}$, we always
 assume that the normal bundle $N_{\Cc{M}/{\Cc{M}_{\text{even}}}}$
 to the even submanifold  ${\Cc{M}_{\text{even}}} \subset \Cc{M}$
is orientable so that one can speak about 
 half densities and the scalar product of
half densities. See   \cite[Remark 2]{MR3307151}.
\color{black}
Let $\B$ denote the Berezinian half-density line bundle over a
 graded manifold $\Cc{M}$, and let $\Gamma(\B)$ be the space of
 global smooth sections, i.e. the space of Berezinian half-densities on $\Cc{M}$. For a fixed coordinate system $(x^a)$ on $\Cc{M}$, a local basis element will be represented by the symbol $\sden$ and a Berezinian half-density takes the expression $\bs = s(x)\sden$, where $s(x)\in C^\infty (\Cc{M})$. As a geometric
 object, a half-density $\bs$ transforms according to the law
    \begin{equation*}
        \bs = s(x)\sden = s(x(x'))\big|\Berr\left(\F{\p x}{\p x'}\right)\big|^{\F{1}{2}}\sqrt{\Cc{D}x'},\qquad x=x(x'),
    \end{equation*}
where $\big|\Berr\left(\F{\p x}{\p x'}\right)\big|$ is the
absolute value of the Berezinian, or superdeterminant, of the 
Jacobian matrix $\left(\F{\p x}{\p x'}\right)$ of the coordinate transformation $x=x(x')$ \cite{MR1202882, MR4235776, MR3307151}.

\begin{example}
Let $\Cc{M}$ be  a $\zz$-graded manifold
 corresponding to a graded vector bundle
$E\to M$.
 That is, the graded manifold with support $M$ where the sheaf $\cA$
 of $\mathbb{Z}$-graded commutative $C^\infty_{M}$-algebras over $M$
is given by $\cA|_U=\Gamma(U, \hat{S}(E^\vee))$,
 for any open  neighborhood $U\subset M$.
Then ${\Cc{M}_{\text{even}}}$ corresponds to the graded vector bundle
$E_{\text{even}}\to M$ of  even degrees. The
condition of  $N_{\Cc{M}/{\Cc{M}_{\text{even}}}}$ being orientable
 is equivalent to requiring that the vector bundle
$E_{\text{odd}}\to M$ of  odd degrees is orientable as
an ordinary vector bundle by forgetting about the degrees.
In this case, the Berezinian line bundle $\text{Ber} (\Cc{M})\to \Cc{M}$
is isomorphic to the pull back bundle of the line bundle 
$\Lambda^{\text{top}} E_{\text{odd}}  \otimes \Lambda^{\text{top}}
E_{\text{even}}^\vee \otimes \Lambda^{\text{top}}T^\vee_M \to M$
via the projection map $\pi: \Cc{M}\to M$ \cite{MR2275685},
while  the Berezinian half-density line bundle $\B$
is isomorphic to the pull back bundle of the line bundle
$(\Lambda^{\text{top}} E_{\text{odd}})^{\F{1}{2}}  \otimes
|\Lambda^{\text{top}} E_{\text{even}}^\vee  \otimes
\Lambda^{\text{top}}T^\vee_M|^{\F{1}{2}} \to M$
via the projection map $\pi: \Cc{M}\to M$.
Here  $|\Lambda^{\text{top}} E_{\text{even}}^\vee  \otimes
\Lambda^{\text{top}}T^\vee_M|$ denotes the line bundle 
$\Lambda^{\text{top}} E_{\text{even}}^\vee  \otimes
\Lambda^{\text{top}}T^\vee_M$ tensored  by its orientation bundle.
Since $E_{\text{odd}}\to M$ is  orientable,
 $\Lambda^{\text{top}} E_{\text{odd}}$ can be identified with
 $|\Lambda^{\text{top}} E_{\text{odd}}|$
and hence $(\Lambda^{\text{top}} E_{\text{odd}})^{\F{1}{2}} $
is defined, by abuse of notation. 
\end{example}

\begin{definition}
A differential operator acting on Berezinian half-densities, of order  
 $\leq n\in\mathbb{Z}$, 
 is an endomorphism of sections
    \begin{equation*}
        \Delta:\Gamma(\B)\rightarrow \Gamma(\B)
    \end{equation*}
such that:
\begin{itemize}

\item for any $n<0$, all operators are identically zero;
\item for $n=0$, any operator $\Delta$ is identified with a function $f\in\ci(\Cc{M})$ which acts through the $\ci(\Cc{M})$-module structure; and
\item for all $n>0$, the commutator $[\Delta,f]$ is a differential operator
 of order $\leq  n-1$ for any operator $f$ of zero order.
\end{itemize}
\end{definition}

Under a chosen coordinate system $(x^a)$, a differential operator of order 
$\leq n$ is determined by a basis of partial derivatives:
	\begin{equation}\label{Section 2 Eqn Local nth orer differential operator}
	   \Delta\bs = \sum^n_{k=0}\F{1}{k!}\Delta^{a_1\cdots a_k}(x)\p_{a_k}\ldots\p_{a_1}s(x)\sqrt{\Cc{D}x},\qquad \bs\in\Gamma(\B).
	\end{equation}

The space of all differential operators on Berezinian half-densities is  denoted by $\DO(\B)$, and carries an increasing filtration determined by the order
 of  differential operators:
     \begin{equation}\label{Section 2 Eqn Filtration by order of diff ops}
        \DO^0(\B)\subset\DO^{\leq 1}(\B)\subset\DO^{\leq 2}(\B)\subset \cdots,\qquad
     \end{equation}
where $\DO^{\leq n}(\B)$ denotes the space of differential operators of order 
at most $n$. In particular, $\DO^0(\B)\cong\ci(\Cc{M})$,
 where functions act by the $\ci(\Cc{M})$-module structure, while $\DO^{\leq 1}(\B)$ is naturally identified with $\Gamma(T_\Cc{M})\oplus\ci(\Cc{M})$,
where vector fields act by taking the Lie derivative. 

It is clear that $\DO(\B)$ is a filtered 
associated algebra when being equipped with the natural multiplication
of composition of operators:
$$        \DO^{\leq n}(\B) \cdot \DO^{\leq m}(\B)\subseteq\DO^{\leq n+m}(\B).
$$
When being equipped with the  commutator of endomorphisms, $\DO(\B)$ is a filtered Lie algebra:
    \begin{equation}
        \label{eq:Cambridge}
        \left[\DO^{\leq n}(\B),\DO^{\leq m}(\B)\right]\subseteq\DO^{\leq n+m-1}(\B)
    \end{equation}
reducing the order by one.
Define the associated graded algebras
    \begin{gather*}
        \Gr\DO(\B) = \bigoplus_{n\geq0}\Gr^n\DO(\B), \quad \text{ and} \label{eq:ERW1}\\ 
\widehat{\Gr}\DO(\B)=\prod_{n\geq0}\Gr^n\DO(\B) \label{eq:ERW2}
    \end{gather*}
where 
$$\Gr^n\DO(\B):=
            \DO^{\leq n}(\B)/\DO^{\leq n-1}(\B).$$

The following is an extension, in the graded context,  of   a well-known
result for ordinary manifolds.
For completeness, we will sketch a proof.
\color{black}

\begin{lemma}
\begin{itemize}
\item[(1).] As a $\ci(\Cc{M})$-module, $\Gr^n\DO(\B)$
 can be naturally identified with $\Gamma({S}^n T_{\Cc{M}})$.
\item[(2).] Under the identification above, 
$\Gr\DO(\B)$ and $\widehat{\Gr}\DO(\B)$ can be naturally identified,
as  $\ci(\Cc{M})$-modules,
with $ \Gamma({S}T_\Cc{M})$ and $ \Gamma(\hat{S}T_\Cc{M})$, respectively.
\end{itemize}
\end{lemma}
\begin{proof}
{(1).}
Introduce the $n$th principal symbol map 
\begin{equation*}
\sigma_n:\DO^{\leq n}(\B)\rightarrow \Gamma(S^nT_{\Cc{M}})
\end{equation*}
as follows. Locally, write $\Delta=\sum^n_{k=0}\F{1}{k!}\Delta^{a_1\cdots a_k}(x)\p_{a_k}\ldots\p_{a_1}$, as in
 Eq. \eqref{Section 2 Eqn Local nth orer differential operator}. Then
 $\sigma_n(\Delta)$ is a
weight $n$ symmetric contravariant tensor field
on $\Cc{M}$ of equal degree, given  by
 \begin{equation}
\label{eq:sigman}
        \sigma_n(\Delta) = \sigma_n\left(\sum^n_{k=0}\F{1}{k!}\Delta^{a_1\cdots a_k}(x)\p_{a_k}\ldots\p_{a_1}\right) = \F{1}{n!}\Delta^{a_1\cdots a_n}(x)\p_{a_n} \odot \ldots \odot  \p_{a_1}.
    \end{equation}
It is simple to check that
$\sigma_n$ 
is indeed a well-defined $\ci(\Cc{M})$-module map, and  moreover, 
any operator  $\Delta$
 is of order $< n$ if and only if  $ \sigma_n(\Delta)$
 vanishes identically. Thus it follows that $\ker  \sigma_n
\cong \DO^{\leq n-1}(\B)$ and therefore the conclusion follows.

(2). This is an immediate consequence of (1).
\end{proof}

Again the following lemma extends a standard result for
ordinary differential operators to the graded context.
\color{black}

\begin{lemma}\label{Section 2 Prop Principal symbol homomorphism}
Let  $\Delta\in \DO^{\leq n}(\B)$ and $\Delta'\in \DO^{\leq m}(\B)$.
  Then
\begin{itemize}
\item[(1).]  
\begin{equation}
\label{eq:Lecce}
\sigma_{n+m}\left(\Delta\cdot \Delta'\right) = \sigma_n(\Delta)\sigma_m(\Delta')\quad \quad \text{and}
\end{equation}
\item[(2).]
    \begin{equation}\label{Section 2 Eqn Principal symbol Lie alg homomorphism}
        \sigma_{n+m-1}\left([\Delta,\Delta']\right) = \allbracket\sigma_n(\Delta),\sigma_m(\Delta')\arrbracket,
    \end{equation}
where $\allbracket \cdot, \cdot \arrbracket$ is the  
degree $0$ Poisson bracket on $\Gamma (ST_\cM)$
 as given by Eq.  \eqref{eq:Manchester}.
%
\end{itemize}
\end{lemma}
\begin{proof}
\color{black}
(1). It suffices to prove Eq. \eqref{eq:Lecce} locally.
 The operators  $\Delta$ and $\Delta'$  can be written locally as follows:
\begin{eqnarray*}
\Delta&=&\sum^n_{k=0}\F{1}{k!}\Delta^{a_1\cdots a_k}(x)\p_{a_k}\ldots\p_{a_1}
\quad \quad \text{and}\\
\Delta'&=&\sum^m_{k=0}\F{1}{k!}(\Delta')^{a_1\cdots a_k}(x)\p_{a_k}\ldots\p_{a_1}.
\end{eqnarray*}
Then $\Delta\cdot \Delta'\in \DO^{\leq n+m}(\B)$
and its leading term is equal to
$$ 
 \F{1}{n!m!} 
(-1)^{|(\Delta')^{b_1\cdots b_m}|(|\p_{a_n}|+\cdots+|\p_{a_1}|)}
\Delta^{a_1\cdots a_n}(x) (\Delta')^{b_1\cdots b_m}(x)
\p_{a_n}\ldots\p_{a_1} \p_{b_m}\ldots\p_{b_1} .
$$
Therefore,
\begin{eqnarray*}
&& \sigma_{n+m} \left(\Delta\cdot \Delta'\right)\\
&=&\F{1}{n!m!} (-1)^{|(\Delta')^{b_1\cdots b_m}|(|\p_{a_n}|+\cdots+|\p_{a_1}|)}
\Delta^{a_1\cdots a_n}(x)  (\Delta')^{b_1\cdots b_m}(x)
\p_{a_n} \odot \ldots \odot  \p_{a_1}\odot \p_{b_m} \odot \ldots \odot  \p_{b_1}\\
&=&\sigma_n(\Delta)\sigma_m(\Delta')
\end{eqnarray*}

(2). We prove it by induction on  the pair $(n, m)$.
It follows from a direct verification that
   Eq. \eqref{Section 2 Eqn Principal symbol Lie alg homomorphism} holds in the case that $n=0, 1$ and $m=0, 1$.
Assume that  Eq. \eqref{Section 2 Eqn Principal symbol Lie alg homomorphism}
holds for $n\leq N$ and $m\leq M$ with $N\geq 1$ and $M\geq 1$.
 We will prove that Eq. \eqref{Section 2 Eqn Principal symbol Lie alg homomorphism}
 holds for $n\leq N$ and $m= M+1$.
Consider the case that  $\Delta'=\Delta'_1 \cdot \Delta'_2$,
where $\Delta'_1 \in \DO^{\leq m_1}(\B)$ and $\Delta'_2 \in \DO^{\leq m_2}(\B)$
are of homogeneous degree, and
  $m=M+1=m_1+m_2$ with  $m_1\leq M$ and $m_2\leq M$.
Then 
\begin{eqnarray*}
&&\sigma_{n+m-1} \left([\Delta,  \Delta']\right)\\
&=&\sigma_{n+m-1} \left([\Delta,  \Delta'_1 \cdot \Delta'_2]\right)\\
&=&\sigma_{n+m-1} \left( [\Delta,  \Delta'_1]\cdot \Delta'_2
+(-1)^{|\Delta||\Delta'_1|} \Delta'_1 \cdot [\Delta, \Delta'_2]\right) \quad
\text{(by Eq. \eqref{eq:Lecce})}\\
&=&\sigma_{n+m_1 -1}\left( [\Delta,  \Delta'_1]\right)\cdot 
\sigma_{m_2}(\Delta'_2)+(-1)^{|\Delta||\Delta'_1|}\sigma_{m_1}(\Delta'_1 )
\sigma_{n+m_2 -1} \left( [\Delta,  \Delta'_2]\right)\  \text{(by induction
assumption)}\\
&=&\allbracket\sigma_n(\Delta),\sigma_{m_1}(\Delta'_1) \arrbracket
\sigma_{m_2}(\Delta'_2) +
(-1)^{|\Delta||\Delta'_1|}\sigma_{m_1}(\Delta'_1)\allbracket\sigma_n(\Delta),\sigma_{m_2}(\Delta'_2) \arrbracket\\
&=&\allbracket\sigma_n(\Delta), \sigma_{m_1}(\Delta'_1)\sigma_{m_2}(\Delta'_2) \arrbracket \quad \quad \text{(by Eq. \eqref{eq:Lecce})}\\
&=&\allbracket\sigma_n(\Delta), \sigma_{m_1 +m_2} (\Delta'_1 \cdot \Delta'_2)
\arrbracket\\
&=&\allbracket\sigma_n(\Delta), \sigma_m (\Delta')\arrbracket .
\end{eqnarray*}
Since any operator in  $\DO^{\leq M+1}(\B)$ can always be written
as such a product $\Delta'_1 \cdot \Delta'_2$ module $\DO^{\leq M}(\B)$,
 we conclude that Eq. \eqref{Section 2 Eqn Principal symbol Lie alg homomorphism}
 holds for any  $n\leq N$ and $m= M+1$.
By induction, it implies that Eq. \eqref{Section 2 Eqn Principal symbol Lie alg homomorphism} holds for any  $n\leq N$ and $m\in \NN$.
Since both sides of Eq. \eqref{Section 2 Eqn Principal symbol Lie alg homomorphism} 
are skew-symmetric,  the conclusion thus follows.
\end{proof}


As a consequence, 
the principal symbol maps establish an isomorphism of degree 
zero Poisson algebras 
$$\widehat{\Gr}\DO(\B)\stackrel{\simeq}{\rightarrow}\Gamma(\hat{S}T_\Cc{M}),$$
 where the bracket on the completed graded algebra $\widehat{\Gr}\DO(\B)$ is induced from the commutator of differential operators, and the bracket on $\Gamma(\hat{S}T_\Cc{M})$ is the canonical Poisson bracket \eqref{eq:Manchester}. 

The existence of a canonical non-degenerate scalar product $\langle-,-\rangle$ 
on half-densities allows us to speak naturally about the (formal) adjoint of 
a differential operator as in the classical case \cite{MR2273508}.

\begin{definition}
Let $\Delta\in \DO^{}(\B)$ be a differential operator
of homogeneous degree. The  (formal) adjoint operator $\Delta^+\in \DO^{}(\B)$ is defined by the relation
	\begin{equation*}
	   \langle \Delta\bs,\bt\rangle = (-1)^{|\Delta||\bs|}\langle \bs,\Delta^+\bt\rangle,
	\end{equation*}
where $\bs,\bt\in\Gamma(\B)$ are compactly supported half-densities of
 homogeneous degrees.
\end{definition}

As in the classical case of ordinary differential operators
\cite[Section 2.1]{MR2273508} \cite[Section VIII]{MR751959},
one can see  that 
for any operator $\Delta$, the (formal) adjoint operator always exists and is unique.

The following lemma can be verified directly.

\begin{lemma}
    \label{lem:CDG} 
    \begin{enumerate} 
    \item Taking the (formal) adjoint is a linear operation on $\DO(\B)$.
    \item If $\Delta$ is a differential operator of order $n$ and of homogeneous degree, then the (formal) adjoint operator $\Delta^+$ is also of order $n$ and carries the same degree.
    \item For any operator $\Delta$, the (formal) adjoint is an involution, i.e.
        $$(\Delta^{+})^+ = \Delta.$$
    \item  For any pair of operators $\Delta$ and $\Delta'$ of homogeneous
 degrees,
        $$(\Delta\cdot \Delta')^+ = (-1)^{|\Delta||\Delta'|}{\Delta'}^+\cdot
\Delta^+.$$
\end{enumerate}
\end{lemma}

\begin{lemma}
\label{lem:example2.5}
Let $\Delta$ be a first order differential operator on $\B$ of 
homogeneous degree,
 described locally by $\Delta = \aDelta^a(x)\p_a + \aDelta_0(x)$. The adjoint operator is given by
	\begin{equation}\label{Section 3 Eqn Adjoint of a first order operator}
	\Delta^+ = -\aDelta^a(x)\p_a - (-1)^{|\p_a||\aDelta^a|}\p_a(\aDelta^a)(x)+\aDelta_0(x).
	\end{equation}
\end{lemma}
\begin{proof}
By  integration by parts, we have 
 $$\langle \p_a\bs,\bt\rangle  =- (-1)^{| \p_a||\bs|}\langle \bs, \p_a \bt\rangle .
$$
It thus follows that 
\begin{equation}
\label{eq:HK1}
\p_a^+=-\p_a .
\end{equation}
On the other hand, 
\begin{equation*}
 \langle  \aDelta^a (x) \cdot \bs,\bt\rangle  = (-1)^{|  \aDelta^a||\bs|}\langle \bs,   \aDelta^a (x) \cdot \bt\rangle .
\end{equation*}
Thus it follows that
 \begin{equation}
\label{eq:HK2}
\aDelta^a(x)^+=\aDelta^a(x) ,
\end{equation}
where, by abuse of notation, $\aDelta^a (x)$ means the multiplication 
operator by $\aDelta^a(x)$.

Therefore, we have
\begin{eqnarray*}
\Delta^+ &=&\big( \aDelta^a(x)\p_a + \aDelta_0(x) \big)^+\\
&=&(\aDelta^a(x)\p_a)^+  + (\aDelta_0(x) )^+ \quad \quad (\mbox{by Lemma } \ref{lem:CDG} (4))\\
&=&(-1)^{|\p_a||\aDelta^a|}\p_a^+\cdot  \aDelta^a(x)^+  +\aDelta_0(x)^+
\quad \quad (\mbox{by Eqs.} \eqref{eq:HK1}-\eqref{eq:HK2})\\
&=&-(-1)^{|\p_a||\aDelta^a|} \p_a \cdot \aDelta^a (x) +\aDelta_0 (x) \\
&=&  -\aDelta^a(x)\p_a - (-1)^{|\p_a||\aDelta^a|}\p_a(\aDelta^a)(x)+\aDelta_0(x).
\end{eqnarray*}
\end{proof}


\begin{definition}
An operator $\Delta\in\DO(\B)$ is called self-adjoint if $\Delta^+ = \Delta$, and anti-self-adjoint if $\Delta^+ = -\Delta$.
\end{definition}

Let $\DO^+(\B)\subset \DO(\B)$ denote the subset
consisting  of all self-adjoint and anti-self-adjoint differential operators on Berezinian half-densities.
Note that although both the space of  self-adjoint differential operators
 and the space of
anti-self-adjoint differential operators are vector
spaces,  $\DO^+(\B)$ is not a vector space.
The sum of a self-adjoint and anti-self-adjoint operator is no longer self or anti-self adjoint.

\begin{lemma}\label{lem:Phi}
For any differential operator $\Delta\in\DO(\B)$  of order $n$,
we have 
\begin{equation}
\label{eq:Phi}
 \sigma_n (\Delta^+ )=(-1)^n \sigma_n  (\Delta ) .
\end{equation}
\end{lemma}
\begin{proof}
    Let $(x^a) $ be coordinates on $\Cc{M}$.
Assume that the leading term of
 $\Delta$ is
$\F{1}{n!} \Delta^{a_{1}\cdots{a_n}}(x)\p_{a_n}\ldots\p_{a_1}$.
According to  Lemma \ref{lem:CDG} (4) and
  Eqs. \eqref{eq:HK1}-\eqref{eq:HK2}, we have
        \begin{equation*}
            \Delta^+ = (-1)^n \F{1}{n!}\Delta^{a_{1}\cdots{a_n}}(x)\p_{a_n}\ldots\p_{a_1} + \cdots.
        \end{equation*}
The conclusion  thus follows immediately.
\end{proof}

\begin{corollary}
\label{Section 2 Eqn Adjoint operator is sign to original}
    Let $\Delta\in\DO^+(\B)$ be a differential operator of order $n$. If $\Delta$ is self-adjoint, then $n$ is necessarily even. Likewise if $\Delta$ is anti-self-adjoint, then $n$ is necessarily odd. Hence, we have
\begin{equation*}
            \Delta^+ = (-1)^n\Delta.
        \end{equation*}
\end{corollary}
\begin{proof}
    Assume that  $\Delta$ is a  self-adjoint operator of order $n$.
Then $\Delta^+=\Delta$. Applying the principal  symbol map,
we have $(-1)^n \sigma_n  (\Delta )=\sigma_n  (\Delta )$ according to
Lemma \ref{lem:Phi}. Since $\sigma_n  (\Delta )\neq 0$,
 thus it follows that $n$ must be even. Similarly
we prove that if $\Delta$ is anti-self-adjoint, then $n$ is necessarily odd.
\end{proof}


\begin{proposition}\label{Section 2 Prop Adjoint ops closed under commutator}
The set $\DO^+(\B)$ is closed under the commutator \eqref{eq:Cambridge}, and 
inherits the increasing filtration
 \eqref{Section 2 Eqn Filtration by order of diff ops} determined by the order.
\end{proposition}

\begin{proof}
That this set inherits the filtration is evident. Now let $\Delta,\Delta'\in\DO^+(\B)$ be operators of order $n$ and $m$ respectively.   Consider
 the commutator
$[\Delta,\Delta']$. From Lemma \ref{lem:CDG} and Corollary \ref{Section 2 Eqn Adjoint operator is sign to original}, it follows that
	\begin{multline}
[\Delta,\Delta']^+ = \left(\Delta\cdot \Delta' - (-1)^{|\Delta||\Delta'|}\Delta'\cdot \Delta\right)^+ = (-1)^{|\Delta||\Delta'|}\Delta^{'+}\cdot \Delta^+ - \Delta^+\cdot \Delta^{'+}\\ 
= (-1)^{|\Delta||\Delta'|+n+m}\Delta^{'}\cdot \Delta - (-1)^{n+m}\Delta\cdot \Delta^{'}
= (-1)^{n+m-1}[\Delta,\Delta']. \label{eq:FRA1}
	\end{multline}
	Since $[\Delta,\Delta']$ is an operator of order at most $n+m-1$,
it must be either  self-adjoint or anti-self-adjoint
 depending on the parity of $n+m-1$.
\end{proof}

Note that although the set $\DO^+(\B)$  is closed under the commutator, it does not form 
a Lie algebra since it is not a vector space.
 
\begin{corollary}
\label{Section 2 Cor vanishing terms in adjoint operator commutator}
Let $\Delta,\Delta'\in\DO^+(\B)$ be differential operators of order $n$ and
 $m$, respectively. Assume that the order of the commutator $[\Delta,\Delta']$ is strictly less than $n+m-1$. Then $[\Delta,\Delta']$ must be of order strictly less than $n+m-2$ also.
\end{corollary}

\begin{proof}
A priori, the commutator $[\Delta,\Delta']$ is an operator of order at most
 $n+m-1$. If we assume that those terms of order $n+m-1$ vanish,
 then the commutator must be an operator of order at most $n+m-2$. 
According to Eq. \eqref{eq:FRA1} in the proof of
 Proposition \ref{Section 2 Prop Adjoint ops closed under commutator},
we have
    \begin{equation*}
        [\Delta, \Delta']^+ = (-1)^{n+m-1}[\Delta,\Delta'],
     \end{equation*}
    and therefore we conclude that the terms of order $n+m-2$ must also
vanish,
 since the parity of the order does not coincide with the adjoint condition. Hence $[\Delta,\Delta']$ is an operator of order at most $n+m-3$.
\end{proof}

\begin{proposition}\label{Section 2 Prop Lie derivative are first order asa}
The space of all first order anti-self-adjoint operators is a  Lie algebra,
which is canonically isomorphic to the Lie algebra of vector fields on
 $\Cc{M}$. Specifically, if $\Delta$ is a first order  anti-self-adjoint
 operator of degree $l$, then it is necessarily the Lie derivative  along
 a degree $l$ vector field on $\Cc{M}$.
\end{proposition}

\begin{proof}
In local coordinates $(x^a)$,  any first order
differential operator of homogeneous degree takes the appearance $\Delta = X^a(x)\p_a + X_0(x)$. The adjoint operator $\Delta^+$ is given by Eq. \eqref{Section 3 Eqn Adjoint of a first order operator}.
Imposing the anti-self-adjoint condition $\Delta = -\Delta^+$ forces
	\begin{equation*}
	\Delta = X^a(x)\p_a + (-1)^{|\p_a||X^a|}\F{1}{2}\p_a (X^a)(x) =
 \Cc{L}_X,
	\end{equation*}
which coincides with the formula for the Lie derivative on  half-densities
$\Gamma (\B)$
 along the  vector field $X=X^a(x)\p_a\in\Gamma(T_{\Cc{M}})$ of the same degree.
\end{proof}

\subsection{$(\hbar)$-enhanced  differential operators}

Introduce a degree $0$ parameter $\hbar$, and define the  space  of $\hbar$-enhanced  differential operators $\hbar\DO(\B)$ by 
    \begin{equation*}
	\hbar \DO(\B):= \DO^{\leq 0}(\B)\oplus \hbar \DO^{\leq 1}(\B)\oplus
\cdots = \prod_{n\geq 0}\hbar^n \DO^{\leq n}(\B),
	\end{equation*}
where $\hbar^n\DO^{\leq n}(\B)$ denotes
the space of  differential operators of order
 at most $n$, acting on Berezinian half-densities enhanced with coefficients in  $\RR \hbarr$.

 An arbitrary
 operator $\Delta\in\hbar\dDO$
is  written as
    \begin{equation}\label{Section 2 Eqn General operator in hDO}
    \Delta = \sum_{n\geq 0}\hbar^n\Delta_{n},\qquad \qquad  \Delta_{n}\in\DO^{\leq n}(\B),
    \end{equation}
where the maximum possible order of each operator
 $\Delta_{n}$ with coefficient $\hbar^n$ is $n$.

The introduction of the parameter $\hbar$ allows for the definition of a secondary filtration on $\hbar\dDO$. Let $\Delta\in\hbar\dDO$
 be an arbitrary operator as defined in 
 \eqref{Section 2 Eqn General operator in hDO}, and define the non-negative integer
    \begin{equation}\label{Section 2 Equ Integer t}
    t(\Delta) := \min_{n\geq 0}\Big\{n - \text{order}(\Delta_n)\, \big|\,
\Delta = \sum_{n\geq 0}\hbar^n\Delta_{n}\Big\}.
    \end{equation}
For example, if any one of the operators $\Delta_{n}$ in the expansion of $\Delta$ were of order $n$, then $t(\Delta) = 0$. For each $t \geq 0$,  define the set of operators
    \begin{equation*}
    \hbar\DO_t (\B) := \Big\{\Delta\in\hbar\dDO\, \big|\, t(\Delta) \geq t\Big\}.
    \end{equation*}

In particular, $\hbar\DO_0 (\B)$ is the whole set $\hbar\dDO$,
whilst for an arbitrary integer
 $t>0$, $\hbar\DO_t (\B)$ consists of those operators $\Delta\in \hbar\dDO$ such that $\hbar^{-t} \Delta\in \hbar\DO_0 (\B)$. This defines a decreasing filtration on $\hbar\dDO$ indexed by $t$:
\begin{equation}
\label{eq:hDO}
\cdots \subset \hbar\DO_t (\B) \subset \cdots \subset \hbar\DO_1 (\B)
 \subset \hbar\DO_0 (\B) = \hbar\dDO.
    \end{equation}

It is beneficial to define a commutator modified by the introduced parameter. Set the modified commutator $[-,-]_\hbar$ on $\hbar\dDO$ as the degree $0$ operation
\begin{equation}
\label{eq:BRU}
 \left[\Delta,\Delta'\right]_\hbar:= \frac{1}{\hbar}
[\Delta, \Delta']
=\sum_{n\geq 0} \hbar^n \sum_{ \substack{i+j=n+1,\\ i, j \geq 0}}[\Delta_{i},\Delta_{j}],
\end{equation}
with $\Delta = \sum_{i\geq 0}\hbar^i\Delta_{i}$ and
 $\Delta' = \sum_{j\geq 0}\hbar^j\Delta'_{j}$, and 
$[-, -]$ denotes the usual commutator.

\begin{remark}
The parameter $\hbar$ should serve as a counting variable, relating the order of an operator to the weight
 of the principal symbol as a contravariant tensor. Since the commutator of two operators reduces the order, it is natural to reduce the power of $\hbar$. 
Note that 
in \cite{MR3654355},
$\hbar$ is introduced as a degree $2$ parameter. This shift in degree
 compensates the degree shift between  $(-1)$-shifted derived Poisson 
structures and $(+1)$-shifted derived Poisson structures.
Recently, Shemyakova also  systematically studied
$(\hbar)$-enhanced  differential operators in connection
with the  study of the $\BV_\infty$  operator
 generating higher Koszul brackets on differential forms \cite{MR4235776}.
\end{remark}

The following lemma is quite obvious.

\begin{lemma}\label{Section 2 Lemma Regular ops as degree 0 reg ops}
\begin{enumerate}
\item For any integer
 $t\geq 0$ and $\Delta\in\hbar\DO_t (\B)$, there exists a unique $\Delta'\in\hbar\DO_0 (\B)$ such that $\Delta = \hbar^t\Delta'$.

\item For any integers
 $s, t\geq 0$, the modified commutator respects the $\hbar$-induced filtration
 $$[\hbar\DO_s (\B),\quad \hbar\DO_t (\B)]_\hbar\subset \hbar\DO_{s+t} (\B). $$
\end{enumerate}
\end{lemma}

\begin{corollary}
\label{cor:BRU}
The space of $\hbar$-enhanced  differential operators $\hbar\dDO$ is a filtered Lie algebra
under  the modified commutator $[-,-]_\hbar$ and the filtration  \eqref{eq:hDO}.
\end{corollary}


\begin{definition}
The extended principal symbol map is defined as
    \begin{equation}
\label{Section 2 Defn eqn Zero extended symbol map definition}
    \sigma_\hbar :\hbar\dDO \rightarrow \Gamma(\hat{S}T_{\Cc{M}})\llbracket\hbar\rrbracket, \qquad \sigma_\hbar(\Delta)
 = \sum_{n\geq 0}\hbar^n\sigma_n(\Delta_{n}),\qquad
    \end{equation}
for  any $\Delta=\sum_{n\geq 0}\hbar^n\Delta_{n}\in \hbar\dDO$,
where $\sigma_n$ is the $n$th principal  symbol map as in
\eqref{eq:sigman}, and
$\Gamma(\hat{S}T_\Cc{M})\llbracket\hbar\rrbracket$ denotes
the space of formal power series in $\hbar$ with coefficients in the 
completed symmetric algebra.
\end{definition}

The following lemma follows immediately from Eq. \eqref{eq:BRU} and
Lemma \ref{Section 2 Prop Principal symbol homomorphism}.

\begin{lemma}
\label{lem:NAP}
For any pair of operators $\Delta, \Delta'\in \hbar\dDO $,
\begin{equation}
\sigma_\hbar \big( \left[\Delta,\Delta'\right]_\hbar\big)=
\F{1}{\hbar}\allbracket  \sigma_\hbar (\Delta), \sigma_\hbar (\Delta')\arrbracket.
\end{equation}
\end{lemma}

Notice that for any operator $\Delta = \sum_{n\geq0}\hbar^n\Delta_n$, the 
extended principal  symbol map is  non-zero only on those terms
 where the order of the operator $\Delta_{n}$ is exactly equal to $n$,
 and hence cannot see symbols of operators of lower orders. To remedy this,
 one can define a sequence of maps, viewing $\hbar\dDO$ as $\hbar\DO_0 (\B)$
 in the decreasing $\hbar$-induced filtration 
 \eqref{eq:hDO}, which extend to each $\hbar\DO_t (\B)$ in turn. For each $t\geq 1$, define
 the $t$-th extended principal symbol map
    \begin{equation}\label{Section 2 Defn nth extended symbol maps}
    \sigma_\hbar^{t}:\hbar\DO_t (\B)\rightarrow\Gamma(\hat{S}T_\Cc{M})\llbracket\hbar\rrbracket,\qquad \sigma_\hbar^{t}(\Delta) = \hbar^{t}\sigma_\hbar(\hbar^{-t}\Delta),
    \end{equation}
for $\Delta\in\hbar\DO_t (\B)$. The extended principal symbol map 
 \eqref{Section 2 Defn eqn Zero extended symbol map definition}
 fits into this sequence with $t = 0$.


\begin{corollary}
For operators $\Delta\in\hbar\DO_t (\B)$, and
 $\Delta'\in\hbar\DO_s (\B)$, the sequence of extended symbol maps preserve the modified commutator
    \begin{equation}\label{Section 2 Eqn Extended symbol Lie algebra property}
    \sigma^{t+s}_\hbar\big([\Delta, \Delta']_\hbar\big) = \F{1}{\hbar}\allbracket \sigma^{t}_\hbar(\Delta),\sigma^{s}_\hbar( \Delta')\arrbracket.
    \end{equation}
\end{corollary}


\begin{definition}
\label{Section 2 Defn Adjoint operator}
An operator $\Delta =\sum_{n\geq 0}\hbar^n\Delta_{n}\in\hbar\dDO$ is called
self-adjoint if $\Delta_n$ is self-adjoint for all even $n$, and anti-self-adjoint for all odd $n$.
\end{definition}

\begin{remark}
If we introduce an adjoint operation on the formal parameter $\hbar$ such
that $\hbar^+=-\hbar$, then the self-adjointness of
 an operator $\Delta\in \hbar\dDO$ can be simply  written as
$$\Delta^+=\Delta. $$
\end{remark}

Denote the set of all $\hbar$-enhanced
self-adjoint operators by $\hbar\DO^+ (\B)\subset\hbar\dDO$, which inherits both
 filtrations 
 \eqref{Section 2 Eqn Filtration by order of diff ops} defined by order,
 and   \eqref{eq:hDO} defined by the introduction of $\hbar$.

\begin{lemma}
\label{lem:NorthWales}
The space of self-adjoint differential operators $\hbar\DO^+ (\B)$ is a
 Lie subalgebra of $\hbar\dDO$ under the modified commutator \eqref{eq:BRU}.
\end{lemma}

\begin{proof}
The space of self-adjoint operators is clearly
a vector space.
 Assume  that $\Delta = \sum_{n\geq 0}\hbar^n\Delta_{n}$ and
 $\Delta' = \sum_{n\geq 0}\hbar^n\Delta'_{n}$ are
self-adjoint differential operators.
By Eq. \eqref{eq:BRU},
 $\left[\Delta,\Delta'\right]_\hbar=
\sum_{n\geq 0} \hbar^n \widetilde{\Delta}_n$,
where  $\widetilde{\Delta}_n= \sum_{i+j=n+1, i, j \geq 0}[\Delta_{i},\Delta_{j}] \in \DO^{\leq n} (\B)$. 
By assumption, for all $n\geq 0$,
$\Delta_{n}$ and $\Delta'_{n}$ are operators
satisfying $\Delta^+_{n} =(-1)^n \Delta_{n} $ and
$(\Delta'_n)^+ =(-1)^n \Delta'_n$. It thus follows
that
$$\widetilde{\Delta}^+_n = \sum_{ i+j=n+1, i, j \geq 0}
[\Delta_{i},\Delta_{j}]^+= \sum_{ i+j=n+1, i, j \geq 0}(-1)^n [\Delta_{i},\Delta_{j}]
=(-1)^n \widetilde{\Delta}_n.$$
Here we used  essentially the same computation as in
 \eqref{eq:FRA1} to prove that $[\Delta_{i},\Delta_{j}]^+
=(-1)^n [\Delta_{i},\Delta_{j}]$.
Therefore $\left[\Delta,\Delta'\right]_\hbar\in  \hbar\DO^+ (\B)$.
Hence, $\hbar\DO^+ (\B)$ is indeed a Lie subalgebra of $\hbar\dDO$.
%
%
\end{proof}

\begin{lemma}
\label{lem:Bari}
Let $\Delta\in\hbar\DO^+_0 (\B)$. Then  $\sigma_\hbar (\Delta)=0$  is equivalent
to that  $\Delta\in\hbar\DO_2^+ (\B)$.
\end{lemma}
\begin{proof}
Let  $\Delta = \sum_{n\geq 0}\hbar^n\Delta_{n}$. Since
 $\sigma_\hbar (\Delta)=0$, it follows
from  \eqref{Section 2 Defn eqn Zero extended symbol map definition}
that $\sigma_n (\Delta_{n})=0$, for all $n\geq 0$.
Hence the order of $\Delta_{n}$ is at most $(n-1)$.  By assumption
$\Delta_{n} \in\DO^+(\B)$  and $\Delta^+_{n}=(-1)^n \Delta_{n}$,
so $\Delta_{n}$ cannot be of order $(n-1)$ according to
 Corollary  \ref{Section 2 Eqn Adjoint operator is sign to original},
and thus at most $n-2$. 
Hence it follows that $\Delta$  indeed  belongs to $\hbar\DO_2^+ (\B)$.
The converse is obvious.
\end{proof}

\begin{corollary}
\label{Section 2 Cor Difference of two adjoint operators}
Let $\Delta$ and $\Delta'\in\hbar\DO^+_t (\B)$ be self-adjoint 
 operators with $t(\Delta) = t = t(\Delta')$  such that
 $\sigma^{t}_\hbar(\Delta) = \sigma^{t}_\hbar(\Delta')$. Then $\Delta-\Delta'\in\hbar\DO^+_{t+2} (\B) $.
\end{corollary}
\begin{proof}
By assumption, we have  $\sigma_\hbar \big( \F{1}{\hbar^t} (\Delta-\Delta')\big)
=0$. According to Lemma \ref{lem:Bari}, $ \F{1}{\hbar^t} 
(\Delta-\Delta')\in \hbar\DO_2^+ (\B)$. Hence $\Delta-\Delta'\in\hbar\DO^+_{t+2}
(\B)$.
\end{proof}


\subsection{$\BV_\infty$-operators and $(-1)$-shifted derived Poisson
 manifolds}

Following Kravchenko \cite{MR1764440}, we introduce the following

\begin{definition}\label{Section 2 Defn BV infinity operator}
A $\BV_\infty$-operator acting on the line bundle $\B$ is a degree $(+1)$
 differential operator $\Delta\in\hbar\DO_0 (\B)$ such that $\sigma_\hbar(\Delta)\neq 0$, and:
\begin{enumerate}
  \item $\lim_{\hbar\rightarrow 0}\Delta = 0$,
  \item $\Delta^2 = \F{\hbar}{2}[\Delta,\Delta]_\hbar = 0$.
\end{enumerate}
\end{definition}
\noindent Condition $(1)$ ensures that $\Delta$ can be expressed as
    \begin{equation*}
    \Delta = \sum_{n\geq 1}\hbar^n\Delta_n,
    \end{equation*}
where each $\Delta_n\in\DO^{\leq n} (\B)$ is a differential operator of order
at most  $n$.

\begin{remark}
Choose  a nowhere vanishing   Berezinian density
 $\brho\in\Gamma(\Ber_\Cc{M})$. With such a choice of Berezinian density,
we may identify differential operators on $\Gamma(\B)$ with differential operators on $\ci(\Cc{M})$ by
    \begin{equation}
\label{eq:Rome}
    \Delta_{\brho}f:=\F{1}{\sqrt{\brho}}\Delta(f\sqrt{\brho}),\qquad f\in\ci(\Cc{M}),\qquad \Delta\in\DO(\B).
    \end{equation}
In this case,  our $\BV_\infty$-operators coincide with the homotopy
 $\BV$-operators of Kravchenko \cite{MR1764440}. See also
\cite{MR4009180, MR3508617, MR3654361, MR3235797, MR2163405, MR2757715, MR3654355, MR4235776, https://doi.org/10.48550/arxiv.2012.14812}.

Since the  Berezinian density line bundle over $\cM$ is
always trivalizable,  in fact, one can always
  work with differential operators
on $\ci(\Cc{M})$ instead of those  on half densities $\Gamma(\B)$. Here,
however, we choose to work with the
latter following \cite{MR2063265, MR3307151, MR2070054, MR2412311}.
  This   approach has the  advantage of being
intrinsic and allows us to speak about the (formal) adjoint
operator of a differential operator naturally, which
plays a crucial role in our construction.  We refer
the interested readers to \cite{MR2063265}
for details on justification regarding the  half densities  approach  as the proper framework
for the Batalin-Vilkovisky (BV) formalism in the classical case.
\end{remark}

For any  $\bX\in\Gamma(\hat{S}T_\Cc{M})$, where $\bX=\sum_{n\geq0}\bX_n$ with
$\bX_n \in \Gamma({S}^nT_\Cc{M})$, define the map
    \begin{equation}\label{Section 2 Eqn hbar map on tensor fields}
        \cdot_\hbar:\Gamma(\hat{S}T_\Cc{M})\rightarrow \Gamma(\hat{S}T_\Cc{M})\hbarr,\qquad \cdot_\hbar(\bX) = \bX_\hbar :=\sum_{n\geq 0}\hbar^n\bX_n.
    \end{equation}
The degree zero Poisson bracket  \eqref{eq:Manchester} on
 $\Gamma(\hat{S}T_\Cc{M})$ then extends to the algebra of \formal  series  $\Gamma(\hat{S}T_\Cc{M})\hbarr$ by
$\hbar$-linearity. The following lemma can be easily verified directly.

\begin{lemma}\label{Section 2 Lemma Bracket of hbar elements}
For any $\bX,\bY, \bZ\in \Gamma(\hat{S}T_\Cc{M})$,
    \begin{equation*}
    \bZ = \allbracket \bX,\bY\arrbracket\quad \Leftrightarrow
\quad \bZ_\hbar = \F{1}{\hbar}\allbracket \bX_\hbar, \bY_\hbar\arrbracket.
    \end{equation*}
\end{lemma}

Essentially, Lemma \ref{Section 2 Lemma Bracket of hbar elements}
indicates that the map  \eqref{Section 2 Eqn hbar map on tensor fields}
 is a morphism of Poisson algebras. Denote
$$ \Gamma(\hat{S}T_\Cc{M})_\hbar :=\{  \bX_\hbar|\bX_\hbar =
\sum_{n\geq 0}\hbar^n \bX_n , \ \text{where } X_n  \in 
 \Gamma({S}^nT_\Cc{M})\} .$$
Then $\Gamma(\hat{S}T_\Cc{M})_\hbar$ is  a Poisson subalgebra of
$\Gamma(\hat{S}T_\Cc{M})\hbarr$ and  the map  \eqref{Section 2 Eqn hbar map on tensor fields}
is an isomorphism of Poisson algebras from
$ \Gamma(\hat{S}T_\Cc{M})$ to $\Gamma(\hat{S}T_\Cc{M})_\hbar$.

The  theorem below indicates that a $\BV_\infty$-operator induces a 
$(-1)$-shifted derived Poisson structure.
This can  be proved using  Voronov's derived
bracket construction \cite{MR2163405}. Below
we will give an alternative proof using Theorem \ref{Section 1 Thm Shifted Poisson structure equiv to formal series}.
Note that this result is known in the literature in
various different forms. See \cite[Theorem 3.2]{MR4235776} and 
\cite[Corollary]{MR3235797},
for instance.

\begin{theorem}\label{Section 2 Theorem BV generates Poisson structure}
If $\Delta\in\hbar\DO_0 (\B)$ is a $\BV_\infty$-operator,
then $C^\infty (\Cc{M})$ together with the sequence  of
multi-brackets  $(\lambda_n)_{n\geq 1}$ defined
by the following equation:
    \begin{equation}
\label{eq:SCE}
    \lambda_n(f_1,\ldots,f_n)\bs:= \lim_{\hbar\rightarrow 0}\left[\cdots\left[\Delta,f_1\right]_\hbar,\ldots,f_n\right]_\hbar\bs,
    \end{equation}
for any functions $f_1,\ldots,f_n\in\ci(\Cc{M})$, and any half-density $\bs\in\Gamma(\B)$,
is a $(+1)$-shifted derived Poisson algebra. Hence $\Cc{M}$ 
 is  a $(-1)$-shifted derived Poisson  manifold.
\end{theorem}

\begin{proof}
Write $\Delta = \sum_{n\geq 1}\hbar^n\Delta_n$, where each
 $\Delta_n\in\DO^{\leq n} (\B)$ is of degree $(+1)$ and order at most $n$.
The extended principal symbol
    \begin{equation*}
\hat{\Pi}_\hbar:=    \sigma_\hbar(\Delta) = \sum_{n\geq 1}\hbar^n \Pi_n,\qquad
\Pi_n= \sigma_n(\Delta_n) \in\Gamma(S^nT_\Cc{M}),
    \end{equation*}
determines a \formal series of degree $(+1)$ symmetric contravariant tensor fields dependent on $\hbar$.
Evaluating at  $\hbar = 1$, we obtain a degree $(+1)$ \formal series
 $\hat{\Pi}:=\hat{\Pi}_\hbar\big|_{\hbar=1}  = \sum_{n\geq 1} {\Pi}_n\in \Gamma(\hat{S}T_\Cc{M})$.

Since the operator $\Delta$ is of degree  $(+1)$,
 its square is expressible via the commutator $\Delta^2 = \F{\hbar}{2}[\Delta,\Delta]_\hbar$. Hence it follows that $[\Delta,\Delta]_\hbar=0$ by definition.

According to Lemma \ref{lem:NAP},
    \begin{equation*}
        \sigma_\hbar\big([\Delta,\Delta]_\hbar\big)=
\F{1}{\hbar}\allbracket  \sigma_\hbar(\Delta), \  \sigma_\hbar(\Delta)\arrbracket
=\F{1}{\hbar}\allbracket \hat{\Pi}_\hbar, \  \hat{\Pi}_\hbar\arrbracket.
    \end{equation*}
From Lemma \ref{Section 2 Lemma Bracket of hbar elements},
 it follows that
$\allbracket \hat{\Pi}, \  \hat{\Pi}\arrbracket=0$.
 One can write $\Pi_1 = Q$, whence $(\Cc{M},Q)$ is a dg manifold equipped with
 a \formal series ${\Pi} = \sum_{n\geq 2}\Pi_n$
satisfying the Maurer-Cartan equation
 \eqref{Section 1 Eqn Maurer-Cartan equation for poisson structure}.
 Hence by Theorem \ref{Section 1 Thm Shifted Poisson structure equiv to formal series}, $\Cc{M}$ is equipped with a  $(-1)$-shifted derived Poisson structure.

To write down the multi-brackets of the corresponding
$(+1)$-shifted derived Poisson algebra,  
we note that the right hand side of \eqref{eq:SCE} is exactly equal to
$\left[\cdots\left[\Delta_n ,f_1\right],\ldots,f_n\right]\bs$.
Since $\Delta_n\in\DO^{\leq n} (\B)$ and $\sigma_n(\Delta_n)=\Pi_n$,
it is standard   \cite{MR2163405}  that the latter is  exactly equal to
$\Pi_n(df_1 \odot \cdots \odot df_n) \bs$.
This concludes the proof of the theorem.
\end{proof}

We note that Eq.  \eqref{eq:SCE}  is essentially Voronov's derived
bracket  formula \cite{MR2163405}.
Recently Bandiera proved the homotopy transfer theorem for
 $\BV_\infty$-algebras \cite{https://doi.org/10.48550/arxiv.2012.14812}.
  In particular, he  showed  that the homotopy transfer of a 
$\BV_\infty$-algebra is compatible with its induced
  $(+1)$-shifted derived Poisson
algebra in the sense of Theorem \ref{Section 2 Theorem BV generates Poisson
 structure}.

\begin{remark}
%
One may compare Eq.  \eqref{eq:SCE} with
 the multi-brackets constructed
 in \cite{MR3654355, MR3235797, MR4235776,  MR3508617, MR3654361}.
As in some of these references  \cite{MR3654355, MR3235797, MR4235776},
we use  the $\hbar$-modification  in order to
absorb   the modified commutator.
Finally, note that to generate a 
 $(+1)$-shifted derived Poisson algebra structure, instead of requiring
$\Delta^2=0$, it suffices to require that 
the extended principal symbol of $\Delta^2$ vanishes,
i.e. $\sigma_\hbar\big( \Delta^2)=0$.
\end{remark}

\begin{example}
Consider the special case of $\BV_\infty$-operators consisting of 
only one term
\begin{equation*}
    \Delta=\hbar^2 \Delta_2,
\end{equation*}
where $\Delta_2$ is a second order differential operator of
 degree $(+1)$ satisfying the condition that $\Delta_2 \cdot  \Delta_2 = 0$.
In this case,  Theorem \ref{Section 2 Theorem BV generates Poisson structure}
implies that $\lambda_n = 0$ for all $n\neq 2$, and the binary bracket $\{-, -\}:=\lambda_2 (-, -)$ is given by,
\begin{multline}
\label{eq:FRA}
\{f, g \} \bs:=\\
\Delta_2 (fg\bs)-(-1)^{(|f|+1)|g|}g \Delta_2
 (f\bs)- (-1)^{|f|} f \Delta_2 (g\bs)+(-1)^{|f|+|g|}fg \Delta_2(\bs),
\end{multline}
for any  $f, g\in C^\infty (\Cc{M})$ of homogeneous degree and $\bs\in
\sections{\B}$.

Assume that there exists a nowhere vanishing
 Berezinian density  $\brho\in\Gamma(\Ber_\Cc{M})$
such that  $\Delta_2 (\sqrt{\brho})=0$. By $\Delta_{\brho}$,
we denote the corresponding second order   differential operator
 on $\ci(\Cc{M})$  as in \eqref{eq:Rome} (with $\Delta$
being replaced by $\Delta_2$), then $\Delta_{\brho}$ is
of degree $(+1)$ and $\Delta_{\brho}^2=0$.
 Eq. \eqref{eq:FRA} implies
that
    \begin{equation}
\label{eq:FLR}
\{f, g\}=  \Delta_{\brho}(fg)-\Delta_{\brho}(f)g-(-1)^{|f|}
f\Delta_{\brho} (g).
\end{equation}
Thus  $\Delta_{\brho}$ is an ordinary
 BV-operator and
$\{\cdot, \cdot\}$ is its corresponding Schouten
bracket. See  \cite{MR837203, MR1256989, MR1952122, MR2070054, MR1906481,  MR1675117, MR2271128}.
\end{example}


\section{Quantization of  $(-1)$-Shifted Derived Poisson Manifolds}

\newcommand{\cqh}{\Cc{Q}_{\hbar}}
\newcommand{\Deltaa}{\hat{\Delta}}
\newcommand{\logeq}{{\mathrel{\raisebox{.66pt}{:}}\Leftrightarrow}}
\newcommand{\MCc}[1]{\MC \left ( #1\right)}
\newcommand{\pbracketh}[2]{\F{1}{\hbar}\allbracket  {#1}, \  {#2}\arrbracket}
\newcommand{\pbracket}[2]{\allbracket  {#1}, \  {#2}\arrbracket}
\newcommand{\lieh}[2]{[#1, \ #2]_\hbar}
\newcommand{\Deltaaa}{{\hat{\Delta'}}}
\newcommand{\Psia}{\Psi}
\newcommand{\argument}{\mathord{\color{black!25}-}}
\newcommand{\gammaa}{\nabla}

\subsection{Quantization of graded manifolds}

Recall that an (affine)
connection on a graded manifold $\Cc{M}$ is a degree zero $\mathbb{R}$-bilinear map $\nabla:\Gamma(T_\Cc{M})\times\Gamma(T_\Cc{M})\rightarrow \Gamma(T_\Cc{M})$ such that
\begin{enumerate}
\item $\nabla_{fX}Z = f\nabla_XZ $;
\item $\nabla_X(fY) = X(f)Y + (-1)^{|f||X|}f\nabla_XY$,
\end{enumerate}
for any  vector fields $X,Y,Z\in\Gamma(T_\Cc{M})$ and smooth functions
 $f,g\in\ci(\Cc{M})$, where $X$ and $f$ are homogeneous.

We say that a connection  $\nabla$ on $\cM$ is torsion-free
if its torsion vanishes. That is,
\[\nabla_X Y -(-1)^{|X| \cdot |Y|} \nabla_Y X - [X,Y]=0 \]
for any homogeneous vector fields $X, Y \in \Gamma(T_\Cc{M})$.
It is well known that  torsion-free connections
always exist for any graded manifold $\Cc{M}$ \cite{MR3910470}.

Fix a torsion-free connection $\nabla$ on $\Cc{M}$,
 and denote the induced linear  connection on the line bundle $\B$ by
the same symbol $\gammaa$.

\begin{definition}
\label{def:Ph}
Let $\Cc{M}$ be a graded manifold with a torsion-free 
 connection $\nabla$. 
Define the map of $\ci(\Cc{M})$-modules
    \begin{equation}\label{Section 3 Eqn The pbw map}
    \Cc{P}_\hbar:\Gamma(\hat{S}T_\Cc{M})\rightarrow \hbar\DO(\B)
    \end{equation}
by the inductive relations
    \begin{gather*}
    \Cc{P}_\hbar(f) = f,\quad f\in\ci(\Cc{M}), \\
    \Cc{P}_\hbar(X) = \hbar\gammaa_X,\qquad X\in\Gamma(T_\Cc{M}),
    \end{gather*}
and for homogeneous elements $X_1,\ldots,X_n\in\Gamma(T_\Cc{M})$,
    \begin{equation}
\label{eq:pbw}
    \Cc{P}_\hbar(X_1\odot\cdots\odot X_n) = \F{\hbar}{n}\sum^n_{k=1}\varepsilon_k\left(\gammaa_{X_k}\Cc{P}_\hbar(\bX^{\{k\}}) - \Cc{P}_\hbar(\nabla_{X_k}\bX^{\{k\}})\right),
    \end{equation}
where $\varepsilon_k = (-1)^{|X_k|(|X_1|+\cdots + |X_{k-1}|)}$ and $\bX^{\{k\}} = X_1\odot\cdots\odot X_{k-1}\odot X_{k+1}\odot\cdots \odot X_n$.
\end{definition}

\begin{proposition}
\begin{enumerate}
\item $\Cc{P}_\hbar$ is  a well defined $\ci(\Cc{M})$-module map;
\item For any   $\bX\in\Gamma(\hat{S}T_\Cc{M})$, we have
\begin{equation*}
    \sigma_\hbar(\Cc{P}_\hbar(\bX)) = \bX_\hbar. 
    \end{equation*}
\end{enumerate}
\end{proposition}
\begin{proof}
Note that $\B$ is a real line bundle over $\cM$.
Hence local flat non-zero sections always exist.
Let $U\subset \cM$ be any contractible   open submanifold and $s\in \Gamma (\B|_U)$ 
a non-zero flat  section with respect to the connection $\gammaa$,
 i.e. $ \gammaa_X s=0$ for any $X\in \Gamma( T_\Cc{M}|_U)$.
Thus we have an isomorphism
\begin{eqnarray*}
 \Phi: \DO(\B|_U ) &\xto{\simeq} D(\cM|_U )\\
\Delta &\xto{} \frac{\Delta (fs)}{s}
\end{eqnarray*}
for any $f\in C^\infty (\cM|_U)$.
It is simple to check that under such an isomorphism,
  the map $\Cc{P}_\hbar$ can be
identified with the PBW map, i.e.
$$ (\Phi\circ \Cc{P}_\hbar) (\bX) =\hbar^n \pbw (\bX), \ \ \forall
\bX\in \Gamma({S}^n T_\Cc{M}|_U), $$
where $\pbw: \Gamma({S}T_\Cc{M})\to D(\cM )$ is the 
PBW map introduced by Liao-Stienon \cite[Definition 3.1]{MR3910470}.
According to \cite[Lemma  3.2]{MR3910470}, $\Phi\circ \Cc{P}_\hbar$ is a
well-defined $C^\infty (U)$-module map
from $\Gamma(\hat{S}T_\Cc{M}|_U)$ to
$ D(\cM|_U )$, which implies that
the inductive formula in Definition \ref{def:Ph}, when
being restricted to $U$, gives rise to a 
well-defined $C^\infty (U)$-module map
from $\Gamma(\hat{S}T_\Cc{M}|_U)$ 
to $\DO(\B|_U )$.
 Therefore it follows that $\Cc{P}_\hbar$ 
is indeed a well-defined $\ci(\Cc{M})$-module map.

To prove the second claim, we can write down
explicitly  the principal
symbol of $(\Phi\circ \Cc{P}_\hbar) (\bX)$ locally.
Below we give a direct proof using induction.
It is clear that the claim is true for all  $ \bX\in \Gamma({S}^{\leq 1} T_\Cc{M})$.
Assume that it is true for any $\bX\in \Gamma({S}^{n-1} T_\Cc{M})$, where
$n\geq 2$. According to Eq. \eqref{eq:pbw}, by induction assumption,
 we have
\begin{eqnarray*}
&& \sigma_\hbar \big(  \Cc{P}_\hbar(X_1\odot\cdots\odot X_n)\big) \\
&=& \F{\hbar}{n}\sum^n_{k=1}\varepsilon_k\left({X_k} \odot
( \sigma_\hbar\Cc{P}_\hbar)
(\bX^{\{k\}}) \right) \\
&=&\F{\hbar}{n}\sum^n_{k=1}\varepsilon_k{X_k} \odot (\hbar^{n-1}\bX^{\{k\}})\\
&=&\hbar^{n}X_1\odot\cdots\odot X_n .
\end{eqnarray*}
Hence, it follows that the claim holds for any $\bX\in \Gamma({S}^{n} T_\Cc{M})$.
This concludes the proof.
\end{proof}

Note that  $ \Cc{P}_\hbar (\bX)$ may not be self-adjoint. In
order to obtain  self-adjoint operators,
 we need to make  some modification.  Introduce a linear map

\begin{equation*}
\cqh:\Gamma(\hat{S}T_\Cc{M})\rightarrow \hbar\DO (\B )
\end{equation*}
as follows. For any $\bX\in \Gamma(\hat{S}T_\Cc{M})$, write
$$  \Cc{P}_\hbar (\bX)=\sum_{n\geq 0}\hbar^n\Delta_{n}\in \hbar\DO (\B). $$
Define
\begin{equation}
\label{eq:cqh}
\cqh  (\bX)=\sum_{n\geq 0}
 \frac{ \hbar^n }{2}\big(\Delta_{n}+(-1)^n\Delta_{n}^+\big).
\end{equation}

\begin{proposition}
\label{pro:chq}
Let $\Cc{M}$ be a $\ZZ$-graded manifold with a torsion-free
  connection $\nabla$. Then $\cqh$ defined  by \eqref{eq:cqh}
is  a degree $0$ map
\begin{equation}
\cqh:\Gamma(\hat{S}T_\Cc{M})\rightarrow \hbar\DO^+ (\B)
\end{equation} 
satisfying the condition that,
for any   $\bX\in\Gamma(\hat{S}T_\Cc{M})$,
\begin{equation*}
\sigma_\hbar ( \cqh (\bX)) = \bX_\hbar.
\end{equation*}
\end{proposition}
\begin{proof}
We have, for any $n\geq 0$, 
$$\big(\Delta_{n}+(-1)^n\Delta_{n}^+     \big)^+
=\Delta_{n}^+ +(-1)^n ({\Delta_{n}^+})^+
= \Delta_{n}^+ +(-1)^n \Delta_{n}=
(-1)^n (\Delta_{n}+(-1)^n\Delta_{n}^+     \big) . $$
Hence, it follows that  $\cqh  (\bX)\in  \hbar\DO^+ (\B)$.
   Moreover,
\begin{eqnarray*}
 \sigma_\hbar( \cqh (\bX))&=&  \sigma_\hbar  \left(
\sum_{n\geq 0} \frac{\hbar^n}{2}\big(\Delta_{n}+(-1)^n\Delta_{n}^+\big) \right)\\
&=& \sum_{n\geq 0} \frac{\hbar^n}{2}\big( \sigma_n (\Delta_{n})
+(-1)^n  \sigma_n (\Delta_{n}^+) \big) \quad (\text{by Eq. \eqref{eq:Phi}})\\
&=& \sum_{n\geq 0} \frac{ \hbar^n}{2}\big( \sigma_n (\Delta_{n})
+\sigma_n (\Delta_{n})\big)\\
&=& \sum_{n\geq 0}\hbar^n\sigma_n(\Delta_{n})\\
&=& \sigma_\hbar( \Cc{P}_\hbar (\bX))\\
&=& \bX_\hbar .
\end{eqnarray*}
This concludes the proof of the proposition. 
\end{proof}

\subsection{Quantization of   $(-1)$-shifted derived Poisson manifolds}

If $\Gamma(\B)$ carries a $\BV_\infty$-operator as defined in
Definition \ref{Section 2 Defn BV infinity operator},
Theorem \ref{Section 2 Theorem BV generates Poisson structure} states that
$\Cc{M}$ inherits a  $(-1)$-shifted derived Poisson structure generated
by the $\BV_\infty$-operator. But one can ask for the converse: when does a 
given  $(-1)$-shifted  derived Poisson manifold admit a $\BV_\infty$-operator which generates
 the derived Poisson structure? 
In this section, we will address this problem. First we will introduce
the following (see also \cite{MR4300181, MR4009180}):


\begin{definition}
\label{Section 4 Defn Quantisation of Poisson structure}
A quantization of a $(-1)$-shifted derived Poisson manifold $(\Cc{M},Q,\Pi)$
 is a $\BV_\infty$-operator $\Delta\in\hbar\DO^+ (\B) $ such that
    \begin{equation*}
    \sigma_\hbar(\Delta)\big|_{\hbar=1} = Q + \Pi,
    \end{equation*}
  where $\sigma_\hbar$ is the extended principal symbol
map as in \eqref{Section 2 Defn eqn Zero extended symbol map definition}.
\end{definition}

\begin{remark}

Classical $\BV$-operators  appeared in the pioneering work
of Batalin--Vilkovisky \cite{MR616572, MR726170, MR750272} in their study of
quantization of gauge theories. The geometry and
the quantization procedure of odd symplectic manifolds were studied
by many authors, in particular, by Schwarz \cite{MR1230027}.
In  a series of papers \cite{MR2063265, MR2412311, MR3685170, MR1952122,
MR2070054, MR3307151},  it was pointed out that  $BV$-operators
can be intrinsically
 realized  as odd Laplace operators acting on half densities,
which are self-adjoint. Our definition above was
motivated by the  consideration  to extend such
an approach to the general $\BV_\infty$-context.
It would also be an interesting question
 to investigate the obstruction
theory to quantization by relaxing the self-adjoint condition
in the definition.

\end{remark}
\color{black}

Note that, for a given dg manifold $(\Cc{M},Q)$, a derived Poisson
structure $\Pi$ is equivalent  to a Maurer-Cartan element
of the dgla 
\begin{equation}
\label{eq:dgla-classical}
\big(\Gamma(\hat{S}T_\Cc{M}), \quad \{\cdot, \cdot\},  \quad\{Q, \cdot\}\big).
\end{equation}
On the other hand, an operator $\Delta\in\hbar\DO^+ (\B) $ is
a  $\BV_\infty$-operator
if and only if $\Delta-\hbarQ$ is  a Maurer-Cartan element
of the  dgla
\begin{equation}
\label{eq:dgla-quantum}
\big( \hbar\DO^+ (\B) ,  \quad[\cdot, \cdot]_\hbar,  \quad[\hbarQ, \cdot]_\hbar\big).
\end{equation}

Summarizing the discussion earlier, we have the following

\begin{theorem}
\label{thm:dgla-exact}
We have the following short exact sequence of dglas
\begin{equation}
\label{eq:dgla-exact}
0\xto{}\hbar\DO^+_2 (\B)\xto{i}\hbar\DO^+ (\B)\xto{\phi}\Gamma(\hat{S}T_\Cc{M})\xto{}0,
\end{equation}
where $\Gamma(\hat{S}T_\Cc{M})$ and $\hbar\DO^+ (\B)$ stand for
the dglas  \eqref{eq:dgla-classical} and \eqref{eq:dgla-quantum}, respectively,
 $\hbar\DO^+_2 (\B)$ is considered as an ideal of  the dgla
 $\hbar\DO^+ (\B)$, the morphism  $i$ is the inclusion map and 
$\phi$ is the morphism defined by
\begin{equation}
 \label{eq:phi1}
\phi (\Delta) = \sigma_\hbar(\Delta)\big|_{\hbar=1}, \quad \quad
\forall \Delta\in \hbar\DO^+ (\B).
\end{equation}
\end{theorem}
\begin{proof}
From Lemma \ref{lem:NAP} and Lemma 
\ref{Section 2 Lemma Bracket of hbar elements}, it follows that
$\phi$ is indeed  a Lie algebra morphism and preserves the differentials,
i.e. a morphism of dglas.
Proposition \ref{pro:chq} implies that $\phi$ is indeed surjective.
Finally, to identify  the kernel of $\phi$, note that,
 for any $\Delta \in \hbar\DO^+ (\B)$, 
the condition that $\phi (\Delta)  =0$ is equivalent to that $\sigma_\hbar(\Delta)=0$.
The latter is equivalent to that $\Delta\in \hbar\DO^+_2 (\B)$
according to Lemma \ref{lem:Bari}. Thus it follows that
$\ker \phi \cong \hbar\DO^+_2 (\B)$. The proof is completed.
\end{proof}

As a consequence, the  quantization problem in Definition
 \ref{Section 4 Defn Quantisation of Poisson structure}
 can be reinterpreted as the lifting problem of
 a    Maurer-Cartan element of the dgla \eqref{eq:dgla-classical}
 to a Maurer-Cartan  element of the dgla \eqref{eq:dgla-quantum},
where both dglas are related by
the short exact sequence  \eqref{eq:dgla-exact}.
As  is standard, we can  find the solution to the latter
by lifting   a sequence of Maurer-Cartan elements as follows.

Note that  we have the following sequence of short exact sequences
of dglas, for all $k=1, 2, \cdots$

\begin{equation}
\label{eq:dgla-exact-k}
0\xto{}\hdo{2k}{2k+2} \xto{i_k} \hdo{}{2k+2} \xto{\prr_k} \hdo{}{2k}\xto{}0,
\end{equation}
where $i_k$ is the natural inclusion and  $\prr_k$ is the natural projection.

As an immediate consequence of Theorem \ref{thm:dgla-exact},
  the morphism $\phi$ in \eqref{eq:phi1} induces an isomorphism of dglas,
 denoted by the same symbol
\begin{equation}
\label{eq:phim}
 \phi: \hdo{}{2}\xto{\simeq} \Gamma(\hat{S}T_\Cc{M})
\end{equation}

Note that the extended principal symbol
 map  \eqref{Section 2 Defn eqn Zero extended symbol map definition}
 descends to a well defined map, for any $k=1, 2, \cdots$,
denoted by the same symbol, by abuse of notation
$$\sigma_\hbar : \hdo{}{2k} \rightarrow \Gamma(\hat{S}T_{\Cc{M}})\hbarr. $$

The following lemma is thus obvious.

\begin{lemma}
\label{lem:Munich}
A    $(-1)$-shifted derived Poisson manifold $(\Cc{M},Q,\Pi)$ is
quantizable if and only if there exists a sequence of
operators $\Deltaa^{(k)}\in  \hbar\DO^+ (\B)$, $k=1, 2, \cdots$,
such that $[\Deltaa^{(k)}]\in \hdo{}{2k}$
is a  Maurer-Cartan element 
and
\begin{equation}
\label{eq:Munich}
\prr_k [\Deltaa^{(k+1)}]=[\Deltaa^{(k)}], 
\end{equation}
 where
\begin{equation}
\label{eq:Munich1}
[\Deltaa^{(1)}]=  \phi^{-1} (\Pi).
\end{equation}
Here 
$[\Deltaa^{(k)}]\in \hdo{}{2k}$ denotes the equivalent class
of  $\Deltaa^{(k)}\in  \hbar\DO^+ (\B)$, 
and $\phi$ is the isomorphism \eqref{eq:phim}. 
\end{lemma}

\begin{remark}
As  is standard, the existence of   the 
sequence of Maurer-Cartan elements
$[\Deltaa^{(k)}]\in \hdo{}{2k}$, where
 $\Deltaa^{(k)}\in  \hbar\DO^+ (\B)$, $k=1, 2, \cdots$,
is not canonical. It depends on how one  chooses the sequence.
 In other words,  for any $n$, the existence of a
Maurer-Cartan element   $[\Deltaa^{(n+1)}]\in \hdo{}{2n+2}$
depends on how $[\Deltaa^{(n)}]\in \hdo{}{2n}$ is picked.
\end{remark}

Note that \eqref{eq:Munich} implies that
$ \sigma_\hbar \big( [\Deltaa^{(k+1)}] \big)= \sigma_\hbar \big( [\Deltaa^{(k)}]
 \big)$,
 and therefore,  by \eqref{eq:Munich1},
we have
$$ \sigma_\hbar \big( [\Deltaa^{(k+1)} ] \big)=\Pi_\hbar, \quad \quad 
k=1, 2, \cdots  .$$
 Thus we have
the following

\begin{lemma}
\label{lem:Kyiv}
Let $\Deltaa^{(k)}\in  \hbar\DO^+ (\B)$, $k=1, 2, \cdots$
be operators as in Lemma \ref{lem:Munich}.
Then, for any $k=1, 2, \cdots$,  we have
\begin{equation}
\label{eq:sigmah}
 \sigma_\hbar \big( [\Deltaa^{(k)} ]\big)=\Pi_\hbar.
\end{equation}
\end{lemma}

Finally, note that for  any  operator $\Deltaa \in  \hbar\DO^+ (\B)$,
we have the following equivalent relation:
\begin{equation}
\label{eq:MCk}
[\Deltaa]\in \MCc{\hdo{}{2k}}\longleftrightarrow \Delta^2 \in \hbar\DO^+_{2k+1} (\B),
\end{equation}
where 
$$\Delta=\hbarQ+\Deltaa .$$

In the sequel, we will use the interpretation of either side
of \eqref{eq:MCk}.

\subsection{Obstruction classes}

In what follows, we will investigate the
obstruction class for quantizing
a   $(-1)$-shifted derived Poisson manifold.
We will prove the following main theorem of the paper.

\begin{theorem}
\label{thm:main}
Let  $(\Cc{M},Q,\Pi)$ be a $(-1)$-shifted derived Poisson manifold.
Assume that the second Poisson cohomology group
 $\Cc{H}^2 (\Cc{M}, Q+ \Pi)$ vanishes, then $(\Cc{M},Q,\Pi)$
is quantizable.
\end{theorem}

According to Lemma \ref{lem:Munich}, the quantization problem
of a $(-1)$-shifted derived Poisson manifold $(\Cc{M},Q,\Pi)$
is equivalent to that of  lifting
a sequence  of Maurer-Cartan elements in the short exact sequences of
dglas \eqref{eq:dgla-exact-k}. Our  strategy is to find the
obstruction class for the lifting in each such a
short exact sequence, which should be of independent
interest.


Assume that  $\Deltaa^{(k)}  \in \hbar\DO^+ (\B)$ such that
 $$  [\Deltaa^{(k)}]\in \MCc{ \hdo{}{2k} } $$
and satisfies  the condition  \eqref{eq:sigmah}.

Let  
$$\Delta^{(k)}=\hbarQ+\Deltaa^{(k)}$$
and 
$$\Omega^{(k)}:= \big(\Delta^{(k)}\big)^2 =\F{\hbar}{2}[\Delta^{(k)},\Delta^{(k)}]_\hbar .$$
Then
\begin{equation}
\label{eq:sigmah1}
 \sigma_\hbar \big( \Delta^{(k)}\big)=\hbar Q+\Pi_\hbar.
\end{equation}

According to \eqref{eq:MCk}, $\Omega^{(k)}\in   \hbar\DO^+_{2k+1} (\B)$
and therefore we have $\frac{\Omega^{(k)}}{\hbar^{2k+1}}\in\hbar\DO^+ (\B)$.
Denote by
$$\bS^{(k)}_{\hbar}=
\sigma_\hbar \big ( \frac{ \Omega^{(k)} }{\hbar^{2k+1}}\big), \quad \quad \text{and }
\quad \bS^{(k)}=\bS^{(k)}_{\hbar}|_{\hbar=1}. $$
By the graded Jacobi identity, we have
$$\lieh{\Delta^{(k)}}{\Omega^{(k)}}=
\lieh{\Delta^{(k)}}{\F{\hbar}{2}[\Delta^{(k)},\Delta^{(k)}]_\hbar}=0. $$ 
Therefore,  by taking the extended principal symbol, we
have
$$ \pbracketh{\hbar Q+\Pi_{\hbar} }{\bS^{(k)}_{\hbar}}=
\pbracketh{\hbar Q+\Pi_\hbar}{\sigma_\hbar (\frac{ \Omega^{(k)}}{\hbar^{2k+1}}
)}
=\sigma_\hbar \lieh{\Delta^{(k)}}{\frac{ \Omega^{(k)}}{\hbar^{2k+1}}}=0 .$$ 
Thus it follows that 
$\pbracket{Q+\Pi}{\bS^{(k)}}=0$, i.e. $\bS^{(k)}$ 
is a $d_\Pi$-cocycle of degree $2$.

Assume that  ${\Deltaaa}^{(k)}  \in \hbar\DO^+ (\B)$ is another
operator representative of $[\Deltaa^{(k)}]$ in $\hdo{}{2k}$.
 Then there exists an operator
$\Psia^{(k)}\in  \hbar\DO^+ (\B)$ such that
$${\Deltaaa}^{(k)}-{\Deltaa}^{(k)}=\hbar^{2k} \Psia^{(k)}. $$
Let
${\Delta'}^{(k)}=\hbarQ+\Deltaaa^{(k)}$.  Therefore
$${\Delta'}^{(k)}={\Delta}^{(k)}+\hbar^{2k} \Psia^{(k)} .$$
Hence we have
$${\Omega'}^{(k)}:= \big({\Delta'}^{(k)}\big)^2 =
\F{\hbar}{2}[{\Delta'}^{(k)},{\Delta'}^{(k)}]_\hbar
={\Omega}^{(k)}+ {\hbar}^{2k+1} \lieh{{\Delta}^{(k)}}{ \Psia^{(k)}}
+\F{\hbar^{4k+1}}{2} \lieh{\Psia^{(k)}}{\Psia^{(k)}}  .$$
Then
$$\F{{\Omega'}^{(k)}}{\hbar^{2k+1}}=\F{{\Omega}^{(k)}}{\hbar^{2k+1}}
+\lieh{\Delta^{(k)}}{\Psia^{(k)}}+ \F{\hbar^{2k}}{2} \lieh{\Psia^{(k)}}{\Psia^{(k)}} . $$
By applying the extended principal  map  $\sigma_\hbar$, we have
$${\bS'}^{(k)}_{\hbar}=\bS^{(k)}_{\hbar}+ \pbracketh{\hbar Q+\Pi_{\hbar} }{
\sigma_\hbar(\Psia^{(k)}} . $$
Hence
$$ {\bS'}^{(k)}=\bS^{(k)}+ \pbracket{Q+\Pi}{Z^{(k)}},$$
where
$ {\bS'}^{(k)}_{\hbar}=\sigma_\hbar \big ( \frac{ {\Omega'}^{(k)} }{\hbar^{2k+1}}\big)$, 
${\bS'}^{(k)}={\bS'}^{(k)}_{\hbar}|_{\hbar=1}$, and
 $Z^{(k)}=\sigma_\hbar(\Psia^{(k)} )|_{h=1}$.
As a result ${\bS'}^{(k)}$ and $\bS^{(k)}$ define
the same class in $\Cc{H}^2 (\Cc{M}, Q+ \Pi)$, independent of the choice of
 the operator representative  in the operator class of $ \hdo{}{2k}$.

\begin{definition}
\label{def:modular}
Let $[\Deltaa^{(k)}]\in \MCc{ \hdo{}{2k} }$ be a Maurer-Cartan element
satisfying  the condition  \eqref{eq:sigmah}.
The class $[\bS^{(k)}]\in \Cc{H}^2 (\Cc{M}, Q+ \Pi)$ 
is called the modular class of $[\Deltaa^{(k)}]$.
\end{definition}

In particular, for $k=1$, we have a canonical element
$[\Deltaa^{(1)}]=  \phi^{-1} (\Pi)\in  \MCc{ \hdo{}{2} }$. Thus
we are led to the following

\begin{definition}
\label{def:modular-Poisson}
For any $(-1)$-shifted derived Poisson manifold  $(\Cc{M}, Q,  \Pi)$,
the modular class of $[\Deltaa^{(1)}]=  \phi^{-1} (\Pi)\in  \MCc{ \hdo{}{2} }$
is called the modular class of the  derived Poisson manifold  $(\Cc{M}, Q,  \Pi)$.
\end{definition}

The modular class in Definition \ref{def:modular-Poisson}
 is an intrinsic invariant of a
  $(-1)$-shifted  derived Poisson manifold.
However, for $k\geq 2$, the  modular class
 $[\bS^{(k)}]\in \Cc{H}^2 (\Cc{M}, Q+ \Pi)$
is not an intrinsic invariant, it depends
on the choice of the Maurer-Cartan element
$[\Deltaa^{(k)}]\in \MCc{ \hdo{}{2k} }$.

\begin{remark}
In the context of supergeometry, the modular class of
 an odd Poisson supermanifold  was originally
 introduced by  Peddie-Khudaverdian in \cite{MR3685170}
 considering odd Laplace-type operators acting on functions via a choice of volume form. The class was further studied in \cite{MR3685170},
 where the class was identified when considering nilpotency conditions of arbitrary second order degree $(+1)$ operators on odd Poisson supermanifolds. 
\end{remark}

\begin{proposition}
\label{pro:Manchester}
Consider the short exact sequence of  dglas \eqref{eq:dgla-exact-k},
$\forall k=1, 2, \cdots$.
Given any  Maurer-Cartan element $[\Deltaa^{(k)}]\in \MCc{ \hdo{}{2k} }$
 satisfying  the condition  \eqref{eq:sigmah},
the modular class
$[\bS^{(k)}]\in \Cc{H}^2 (\Cc{M}, Q+ \Pi)$
 is the obstruction class   of
lifting  $[\Deltaa^{(k)}]$ to a Maurer-Cartan element  in ${ \hdo{}{2k+2}}$.
\end{proposition}
\begin{proof}
Assume that the modular class vanishes. That is, there exists
a degree $(+1)$ element $X^{(k)}\in \Gamma(\hat{S}T_\Cc{M})$ such that
$\bS^{(k)}=\pbracket{ Q+\Pi}{X^{(k)}}$. Hence, by Lemma 
\ref{Section 2 Lemma Bracket of hbar elements}, we have

\begin{equation}
\bS^{(k)}_{\hbar}=\pbracketh{{\hbar} Q+\Pi_{\hbar}}{X^{(k)}_{\hbar}} .
\end{equation}

Choose a  self-adjoint operator $\Phi^{(k)} \in   \hbar\DO^+ (\B)$
satisfying the condition  that $\sigma_{\hbar} \big(\Phi^{(k)}\big)
=X^{(k)}_{\hbar}$. This is always possible according to 
Proposition \ref{pro:chq}.
Let
$$ \Deltaa^{(k+1)}= \Deltaa^{(k)}-\hbar^{2k}\Phi^{(k)} .$$
It is clear that
$$\prr_k [\Deltaa^{(k+1)}]=[\Deltaa^{(k)}] \in \hdo{}{2k}. $$
Here $[\Deltaa^{(k+1)}]$ denotes the class  of
$\Deltaa^{(k+1)}$ in $\hdo{}{2k+2}$, while $[\Deltaa^{(k)}]$
denotes the class  of $\Deltaa^{(k)}$ in $\hdo{}{2k}$.
It remains to prove that $[\Deltaa^{(k+1)}]$ is 
a  Maurer-Cartan element  in ${ \hdo{}{2k+2}}$.
Let $\Delta^{(k+1)}= \hbarQ+\Deltaa^{(k+1)}$.
Then 
$$ \Delta^{(k+1)}= \Delta^{(k)}-\hbar^{2k}\Phi^{(k)} .$$
Now we have
\begin{eqnarray*}
\Omega^{(k+1)}:&=& \big(\Delta^{(k+1)}\big)^2\\
& =&\F{\hbar}{2}[\Delta^{(k+1)},\Delta^{(k+1)}]_\hbar \\
& =&\F{\hbar}{2}[\Delta^{(k)}-\hbar^{2k}\Phi^{(k)}, \ \Delta^{(k)}-\hbar^{2k}\Phi^{(k)}]_\hbar \\
& =&\Omega^{(k)}-\hbar^{2k+1}[\Delta^{(k)}, \Phi^{(k)}]_\hbar
+\F{\hbar^{4k+1}}{2}[\Phi^{(k)}, \Phi^{(k)}]_\hbar .
\end{eqnarray*}
Thus it follows that
$$\F{\Omega^{(k+1)}}{\hbar^{2k+1}}
=\F{\Omega^{(k)}}{\hbar^{2k+1}}
-[\Delta^{(k)}, \Phi^{(k)}]_\hbar
+\F{\hbar^{2k}}{2} [\Phi^{(k)}, \Phi^{(k)}]_\hbar .$$
Hence, by applying the extended principal map $\sigma_\hbar$, we have
\begin{eqnarray*}
\sigma_\hbar \big( \F{\Omega^{(k+1)}}{\hbar^{2k+1}}\big)
&=&\sigma_\hbar \big( \F{\Omega^{(k)}}{\hbar^{2k+1}}\big)
- \pbracketh{ \sigma_\hbar \big( \Delta^{(k)} \big)}{\sigma_\hbar \big( \Phi^{(k)}\big)}\\
&=& \bS^{(k)}_{\hbar}-\pbracketh{{\hbar} Q+\Pi_{\hbar}}{X^{(k)}_{\hbar}}\\
&=&0.
\end{eqnarray*}
Therefore it follows that
 $\Omega^{(k+1)} \in  \DO^+_{2k+3}(\B)$.
According to  the equivalent relation
\eqref{eq:MCk}, $[\Delta^{(k+1)}]$ is
indeed a  Maurer-Cartan element  in ${ \hdo{}{2k+2}}$.
This proves one direction of the proposition.
The converse can be proved using the same argument by going backwards.
\end{proof}

\begin{remark}
For any $k=1, 2, \cdots$, the short exact sequence
\eqref{eq:dgla-exact-k} is in fact a square-zero extension of   dglas.
Recall that a square-zero extension of   dglas is a short exact sequence
of dglas
\begin{equation}
\label{eq:dglaext}
0\xto{}K\xto{i}L\xto{\pi}M\xto{}0
\end{equation}
such that $[ K, \ K ]=0$. One can prove that
for a  square-zero extension of dglas \eqref{eq:dglaext},
 the obstruction to lifting
a  Maurer-Cartan element $x$ in $M$ is a well defined class 
in 
\begin{equation}
\label{eq:2nd}
\Cc{H}^2 \big(K, d_K+ [ y, \ -] \big),
\end{equation}
where $y\in L$ is any lift of $x$, i.e. $\pi (y)=x$. 
Note that the differential $d_K+ [ y, \ -]$,
and hence the
cohomology group $\Cc{H}^2 \big(K, d_K+ [ y, \ -] \big)$ is independent of the choice of the lift $y$.
One  can prove that for the square-zero extension of 
  dglas \eqref{eq:dgla-exact-k}, the corresponding
 cohomology group \eqref{eq:2nd}
is isomorphic to the second Poisson cohomology group 
 $\Cc{H}^2 (\Cc{M}, Q+ \Pi)$. See
\cite{Manetti:book} and \cite[Section 3.1.3]{ MR3370863}
 for  references on a related 
problem (the case when $K$ is an abelian ideal of  $L$).
\end{remark}

For each $\bX\in\Gamma(\hat{S}T_\Cc{M})$,
by $\Op(\bX)$, we denote 
 the set of all self-adjoint  operators
 $\Delta\in\hbar\DO^+ (\B)$ such that $\sigma_\hbar\left(\Delta\right) = \bX_\hbar$, where $\sigma_\hbar$ is the extended principal symbol map  \eqref{Section 2 Defn eqn Zero extended symbol map definition}, and $\bX_\hbar$ is defined as in  \eqref{Section 2 Eqn hbar map on tensor fields}. The subset $\Op(\bX)$ is called the operator class of the \formal series $\bX$.

By considering the particular case of $k=1$, Proposition \ref{pro:Manchester}
implies the following

\begin{corollary}
\label{cor:LUG}
For any  $(-1)$-shifted derived Poisson manifold $(\Cc{M},Q,\Pi)$,
its modular class is the    obstruction class to the existence of
an operator $\Deltaa  \in \Op(\Pi)$ satisfying the condition that
 $\Delta^2\in \hbar\DO^+_5 (\B)$,  where
$\Delta=\hbarQ+\Deltaa$.
\end{corollary} 

Now we are ready for the proof of the main theorem.

\begin{proof}[Proof of Theorem \ref{thm:main}]
It is an immediate consequence  of 
Lemma \ref{lem:Munich}, Definition \ref{def:modular}, and
Proposition \ref{pro:Manchester}.
\end{proof}

\subsection{$(-1)$-shifted dg Poisson manifolds}

Let $\Cc{M}$ be a graded manifold equipped with a
 single degree $(+1)$ Poisson bracket
 $\lambda_2:\ci(\Cc{M})^{\otimes 2}\rightarrow\ci(\Cc{M})$.
This is essentially the classical case \cite{MR3685170}.
The Poisson structure is defined by a degree $(+1)$ element
 $\Pi_2 \in\Gamma(S^2T_\Cc{M})$ such that $\allbracket\Pi_2,\Pi_2\arrbracket =
 0$. 
It is denoted  by $(\cM, \Pi_2)$ and is called
a  $(-1)$-shifted Poisson manifold.

Let  $\Delta_2\in \DO^+(\B)$ be a degree $(+1)$ 
self-adjoint second order differential operator 
 such that $\sigma_2 (\Delta_2)=\Pi_2$.
Let
    \begin{equation*}
  \Delta = \hbar^2\Delta_2.
    \end{equation*}
Since $\Delta_2 \circ \Delta_2=\F{1}{2}[\Delta_2, \ \Delta_2]$
 is a first order anti-self-adjoint operator, 
by
 Proposition \ref{Section 2 Prop Lie derivative are first order asa},
 it is given by the Lie derivative 
\begin{equation}
\label{eq:NYC}
\Delta_2 \circ \Delta_2 =\Cc{L}_{X_1},
\end{equation}
 where $X_1 \in\Gamma(T_\Cc{M})$ is a degree $2$ vector field.
Hence $\F{1}{\hbar^3} \Delta^2 =\hbar \Delta_2 \circ \Delta_2=
\hbar\Cc{L}_{X_1}$.
 Thus  the modular class is $[X_1 ]\in\Cc{H}^2 (\Cc{M},\Pi_2)$.

More generally, if $(\Cc{M}, Q)$ is a dg manifold equipped
with the degree $(+1)$-Poisson bracket determined by
a degree $(+1)$ element $\Pi_2\in\Gamma(S^2T_\Cc{M})$ such that
$$\allbracket\Pi_2, \ \Pi_2\arrbracket = 0, \ \ \ \allbracket Q, \  \Pi_2\arrbracket = 0,$$ 
then $(\Cc{M}, Q, \Pi_2)$ is called  a  \emph{$(-1)$-shifted dg Poisson manifold}.
Let $\Delta_2\in \DO^+(\B)$ be an operator as before, and
$\Delta=\hbar\Cc{L}_Q+\hbar^2 \Delta_2$.
Then
$$\Omega:=\Delta^2=\hbar^3 [\Cc{L}_Q, \Delta_2]+\hbar^4\Delta^2_2
=\hbar^3 [\Cc{L}_Q, \Delta_2]+\hbar^4 \Cc{L}_{X _1}.$$
Therefore 
\begin{equation}
\sigma_h (\F{ \Omega}{\hbar^3})=X_0 +\hbar X_1, \label{eq:X01} 
\end{equation}
where $X_0=\sigma_0 ([\Cc{L}_Q, \Delta_2]) \in C^\infty (\Cc{M})$.
Thus $[X_0 +  X_1 ]\in  \Cc{H}^2 (\Cc{M}, Q+\Pi_2)$
is the modular class of the $(-1)$-shifted   dg  Poisson manifold
$(\Cc{M}, Q, \Pi_2)$.
According to Corollary \ref{cor:LUG},  
it is the obstruction class to the existence of $\Deltaa
 \in \Op(\Pi)$ such that $(\hbarQ+\Deltaa)^2\in \hbar\DO^+_5 (\B)$. 
Note that the vanishing of the modular class
 $[X_0 +  X_1 ]\in  \Cc{H}^2 (\Cc{M}, Q+\Pi_2)$ does not necessarily
mean that $(\cM, Q, \Pi_2)$ is quantizable. However,  if we assume
 furthermore that
$X_0 +  X_1$ can be expressed as a particular type of
coboundary, i.e.
 there exists a function  $f\in C^\infty (\Cc{M})$ such that
\begin{equation}
\label{eq:JFK}
X_0 +  X_1=\{Q+\Pi_2, \ f\} .
\end{equation}
Then 
$$ \Delta'=\hbar\Cc{L}_Q+\hbar^2 (\Delta_2 -f)$$
is clearly a $\BV_\infty$-quantization of  $(\Cc{M}, Q, \Pi_2)$.

\begin{proposition}
\label{pro:EWR}
Let $(\Cc{M}, Q, \Pi_2)$ be a $(-1)$-shifted   dg  Poisson manifold,
$\Delta_2\in \Op(\Pi)$, and 
  $\Delta=\hbar\Cc{L}_Q+\hbar^2 \Delta_2$.
Assume that there exists a function
 $f\in C^\infty (\Cc{M})$ such that Eq. \eqref{eq:JFK} holds
 where $X_0$ and $X_1$ are as in
\eqref{eq:X01}. Then $(\Cc{M}, Q, \Pi_2)$ is quantizable.
\end{proposition}

 In particular, as an immediate consequence, we recover the following
theorem due to Khudaverdian-Peddie.

\begin{corollary}[\cite{MR3685170}]
A  $(-1)$-shifted  Poisson manifold
 $(\Cc{M}, \Pi_2)$ is  quantizable if and only if
its modular class vanishes.
\end{corollary}
\begin{proof}
If $(\Cc{M}, \Pi_2)$ is  quantizable, then its modular class
must vanish according to Corollary \ref{cor:LUG}.
Conversely, assume that the modular class of $(\Cc{M}, \Pi_2)$
vanishes.  By Eq. \eqref{eq:NYC}, the modular cocycle
is a degree $2$ vector field $X_1\in \sections{T_\cM}$. Since it
is a coboundary by assumption, 
by the weight counting,  we have $  X_1=\{\Pi_2, \ f\}$ for
some function  $f\in C^\infty (\Cc{M})$.
According to Proposition \ref{pro:EWR}, $(\Cc{M}, \Pi_2)$ is indeed  quantizable. 
\end{proof}

Below is an example of  $(-1)$-shifted  Poisson manifold with
 non-vanishing modular class, which is adapted from \cite{MR3685170}.

\begin{example}
\label{ex:Hovik}
Let $\Cc{M} = \mathbb{R}^2[-1]\times\mathbb{R}[-2]$, and choose global coordinates $\xi,\tau, z$ with assigned degrees
	\begin{equation*}
 |\xi| = |\tau| = -1,\qquad |z| = -2.
	\end{equation*}
Define two vector fields on $\Cc{M}$:
	\begin{gather*}
	P = \tau\F{\p}{\p\xi} + \tau\xi\F{\p}{\p z}, \qquad Q = \F{\p}{\p\tau}.
	\end{gather*}
Then $|P|=0$ and $|Q|=1$.
It is simple to check that
$$[Q, Q]=0,   \qquad [P, P]=0,  \qquad [P, Q]=0.$$
 Hence $\Pi_2=P\odot Q$ is a  degree $(+1)$ contravariant tensor
satisfying $\allbracket\Pi_2, \ \Pi_2\arrbracket = 0$.
Then
	\begin{equation}
\label{eq:Hovik}
	\Delta_2 = \F{1}{2}
\big(\Cc{L}_P \circ \Cc{L}_Q + \Cc{L}_Q \circ \Cc{L}_P\big) \in \DO^+(\B) 
	\end{equation}
 is a degree $(+1)$
self-adjoint second order differential  operator  such that $\sigma_2 (\Delta_2)=\Pi_2$.
 The operator has the coordinate expression
	\begin{equation*}
	 \Delta_2 = \tau\p_\xi\p_\tau + \tau\xi\p_z\p_\tau + \F{1}{2}\xi\p_z + \F{1}{2}\p_\xi,
	\end{equation*}
from which we can explicitly calculate the square:
	\begin{equation*}
	\Delta_2 \circ \Delta_2 = \F{1}{4}\p_z.
	\end{equation*}	
Hence $X_1=\F{1}{4}\p_z$.  
Note that $\p_z$ cannot be a Hamiltonian vector field, 
since the $(+1)$-shifted Poisson bracket on $\Cc{M} $ reads
	\begin{equation*}
	\{\xi,\tau\} = \pm\tau,\qquad \{z,\tau\} = \pm\tau\xi,
	\end{equation*}
which is always proportional to $\tau$. Hence $[\p_z] \in  \Cc{H}^2 (\Cc{M}, \Pi_2)$
 is a non-trivial modular class of the $(-1)$-shifted Poisson  manifold $(\Cc{M}, \Pi_2)$. As a consequence, $(\cM, \Pi_2)$ is not quantizable.
\end{example}

\begin{example}
As in Example \ref{ex:Hovik}, we consider the $(-1)$-shifted   dg
  Poisson manifold $(\Cc{M}, Q, \Pi_2)$.
Let $\Delta=\hbar\Cc{L}_Q+\hbar^2 \Delta_2$, where
$\Delta_2$ is as in  \eqref{eq:Hovik}.
Then
$$[\Cc{L}_Q, \Delta_2]=\F{1}{2}[\Cc{L}_Q, \ \Cc{L}_P \circ \Cc{L}_Q + \Cc{L}_Q \circ \Cc{L}_P ]
=\Cc{L}_Q \circ \Cc{L}_P \circ\Cc{L}_Q=(\Cc{L}_{[Q, P]} +\Cc{L}_P \circ  \Cc{L}_Q)
\circ \Cc{L}_Q=0 .$$
Therefore   $X_0=\sigma_0 ([\Cc{L}_Q, \Delta_2])=0$,
and $[\p_z] \in  \Cc{H}^2 (\Cc{M}, Q+ \Pi_2)$
 is a non-trivial modular class of the $(-1)$-shifted
dg  Poisson  manifold $(\Cc{M},  Q, \Pi_2)$.
\end{example}

\subsection{$(-1)$-shifted derived Poisson manifolds and $\L$-algebroids}

Following the conventions of \cite{MR2840338, MR2163405, MR1327129, MR4091493},
an $\L$-algebroid is a $\ZZ$-graded vector bundle
 $\Cc{\aV}\rightarrow \Cc{M}$  together with
\begin{itemize}
\item a sequence of multi-liner maps $\lambda_n:\bigwedge^n\Gamma(\Cc{\aV})\rightarrow \Gamma(\Cc{\aV})$ of degree $2-n$, $n\geq 1$, which determine an $\L$-structure on $\Gamma(\Cc{\aV})$; and
\item a sequence of bundle maps $\rho_n:\bigwedge^n\Cc{\aV}\rightarrow T_\Cc{M}$ of degree $1-n$, $n\geq0$, called multi-anchor maps, such that

\begin{itemize}
\item[(i)]
its induced maps $\rho_n :  \bigwedge^n \sections{\Cc{\aV}}\to
\XX (\Cc{M}), \ n\geq 1$
 define an $\L$-morphism
from $\Gamma(\Cc{\aV} )$ to the dgla  $\big(\XX (\Cc{M}), \rho_0)$,
 and  moreover,
\item[(ii)] the higher Leibniz rules hold:
        \begin{equation}
\label{eq:Lebniz}
        \lambda_n(v_1,\ldots,v_{n-1},fv_n) = \rho_{n-1}(v_1,\ldots,v_{n-1})(f)v_n
        + (-1)^{|f|(n+|v_1|+\cdots+|v_{n-1}|)}f\lambda_n(v_1,\ldots,v_n),
        \end{equation}
for all $n\geq 1$,  $v_1,\ldots,v_n\in\Gamma(\Cc{\aV})$ and $f\in\ci(\Cc{M})$.
\end{itemize}

\end{itemize}

When $\cM$ is an ordinary manifold $M$ (being considered
of degree zero), and the vector bundle $\Cc{\aV}=\oplus_{i\geq 0}\Cc{\aV}^i\to
\cM$
is a non-negative graded vector bundle,
this reduces to the  $L_\infty$ algebroids
 studied by
Laurent-Gengoux et. al. \cite{MR4164730}.
In particular, if  $\Cc{\aV}$ is a usual vector bundle
 concentrated in degree $0$, it becomes an  ordinary  Lie
algebroid over $M$. Here, we adapt more general
definition by allowing $\cM$ to be $\ZZ$-graded as well.
For references on $L_\infty$ algebroids, see~\cite{MR4091493, MR2757715, MR2695305, MR1854642,  MR4164730, MR3300319, MR3277952, MR3313214,  MR3631929, MR3584886}
and the references therein.

Similar to the ordinary Lie-Poisson construction, 
any  $L_\infty$-algebroid gives rise to
a $(-1)$-shifted derived Poisson manifold
in a natural fashion. 

The following proposition  is standard. 
See \cite{arxiv.1808.10049, MR4091493, MR4235776}. 

\begin{proposition}
\label{pro:Linfinity-alg}
An $\L$-algebroid structure on $\Cc{\aV}$ is equivalent to a linear 
 $(-1)$-shifted derived Poisson manifold structure on $\Cc{\aV}^\vee[-1]$.
\end{proposition}

Let us explain the word  \emph{linear}.
Note that under the isomorphism:
\begin{equation}
\label{eq:YYZ}
\ci(\Cc{\aV}^\vee[-1]) \cong \sections{ { \hat{S} }^{\bullet}( \Cc{\aV}[1] )},
\end{equation}
 $\ci(\Cc{\aV}^\vee[-1])$,
 besides the underlying $\mathbb{Z}$-grading,
  possesses a natural non-negative
 $\mathbb{N}$-grading by the weight,
 which we denote by $N$-degree. 
A $(-1)$-shifted derived Poisson manifold 
structure on $\Cc{\aV}^\vee[-1]$
is said to be {\em linear}, if   for each $n\geq 1$,
the Poisson multi-bracket $\tilde{\lambda}_n:\ci(\Cc{\aV}^\vee[-1])^{\otimes n}\rightarrow\ci(\Cc{\aV}^\vee[-1])$  is of $N$-degree  $1-n$.
More precisely, under the isomorphism \eqref{eq:YYZ},
$\tilde{\lambda}_n$ is of the form:
\begin{equation}
\label{eq:Scalea}
 \tilde{\lambda}_n : \sections{ { S }^{k_1}( \Cc{\aV}[1] )}\times \cdots\times\sections{S^{k_n }(  \Cc{\aV}[1] )}\to
\sections{S^{k_1+ \cdots k_n +(1-n)}( \Cc{\aV}[1] )}, 
\end{equation}
for any non-negative integers $k_1, \cdots, k_n$.
In this case, we say that $\Pi_n$ is of  $N$-degree  $1-n$.

Below we  describe briefly  one direction of the construction 
in Proposition \ref{pro:Linfinity-alg}. The construction of the other
 direction can be  carried out by going backwards.

Let $ \Cc{\aV}\to \cM$ be an $L_\infty$-algebroid with the
structure maps
 $(\lambda_n)_{n\geqslant 1}$ and $(\rho_n)_{n\geqslant 0}$ as
in the beginning of the subsection. 
Then $\sections{ \Cc{\aV}[1]}$ is an $L_\infty[1]$-algebra.
The corresponding degree $(+1)$-Poisson multi-brackets
$\tilde{\lambda}_n$ as in  \eqref{eq:Scalea} are essentially
the natural extension of $(\lambda_n)_{n\geqslant 1}$ 
 by the Leibniz rule in Definition \ref{def:derived}
 with the help of the multi-anchor maps.
Note that
%
the $0$-th anchor map $\rho_0$ defines a 
homological vector field $Q_\cM\in \XX (\cM)$.
The unary bracket $\lambda_1: \sections{ \Cc{\aV}}\to\sections{ \Cc{\aV}}$ and $\rho_0$
 are compatible:
\[ \lambda_1(fa)=Q_\cM(f)a+(-1)^{|f|}f \lambda_1(a),\quad\forall 
a\in\sections{ \Cc{\aV}},f\in C^\infty (\cM) .\]
Hence $ \sections{ \Cc{\aV}}$ is a dg module over $( C^\infty (\cM), Q_\cM)$.
As a consequence,  $ \Cc{\aV}\to \cM$ is a dg vector bundle \cite{MR2709144, MR3319134},
 which implies that $ \Cc{\aV}^\vee\to \cM$  is a dg vector bundle as well.
Hence ${\Cc{\aV}^\vee[-1]}$ is, naturally,  a dg manifold, whose
corresponding homological vector field is denoted $Q$.
As before, we write $\Pi_1=Q$, and the corresponding
degree $(+1)$ Poisson tensor
  $\Pi\in \sections{ \hat{S}(T_{\Cc{\aV}^\vee[-1]})}$
on ${\Cc{\aV}^\vee[-1]}$ is
        \begin{equation*}
        \Pi = \sum_{n\geq 1}\Pi_n,\quad |\Pi| = 1,\quad N(\Pi_n) = 1-n.
        \end{equation*}

Choose a local coordinate chart $(x^a)$ on $\cM$ and a local frame
 $(e_1,\ldots, e_n)$ on $\Cc{\aV}$. Let $(f^1, \ldots,f^n)$
be its corresponding  dual local frame on $\Cc{\aV}^\vee$.
These data determine a  coordinate chart on ${\Cc{\aV}^\vee[-1]}$:
\begin{equation}
\label{eq:localchart1}
(x^a \ ,\xi_i),
\end{equation}
and  as well as a coordinate chart on ${\Cc{\aV}[1]}$:
\begin{equation}
\label{eq:localchart2}
(x^a \ ,\eta^i).
\end{equation}

 Under the local  coordinate chart \eqref{eq:localchart1},
$\Pi_n$ can be  written explicitly as follows
    \begin{equation}
\label{eq:pi}
        \Pi_n = \frac{1}{(n-1)!}\rho^a_{i_1\cdots i_{n-1}}(x)
 \frac{\p}{\p\xi_{i_1}}\cdots\odot\frac{\p}{\p\xi_{i_{n-1}}} \odot \frac{\p}{\p x^a}
 + \frac{1}{n!}C^j_{i_1\cdots i_n}(x)\xi_j\frac{\p}{\p\xi_{i_1}}\odot \cdots \odot\frac{\p}{\p\xi_{i_n}},
    \end{equation}
where  the structure functions $( \rho^a_{i_1\cdots i_{n-1}} (x) )$ 
and $(C^j_{i_1\cdots i_n}(x))$ are
determined by the corresponding multi-anchors and multi-brackets:
    \begin{equation*}
        \rho_{n-1} (e_{i_1}\wedge\cdots \wedge e_{i_{n-1}}) = \pm \rho^a_{i_1\cdots i_{n-1}}(x)
\frac{\p}{\p x^a} ,\qquad \lambda_n (e_{i_1},\cdots , e_{i_{n}} ) = \pm C^j_{i_1\cdots i_n}(x)e_j.
    \end{equation*}

\begin{theorem}
\label{thm:main2}
For any  $\L$-algebroid $\Cc{\aV}$,
its corresponding linear $(-1)$-shifted   derived  Poisson manifold 
$\Cc{\aV}^\vee[-1]$  admits a canonical quantization.
\end{theorem}
\begin{proof}
To any  $\L$-algebroid $\Cc{\aV}$, there  is 
a Chevalley--Eilenberg differential $D$,
which is a homological vector field on $\Cc{\aV}[1]$.
Indeed, $\L$-algebroid structures on a graded vector bundle
$\Cc{\aV}\to \cM$  are equivalent to homological vector fields
 on $\Cc{\aV}[1]$  vanishing on the zero section \cite{MR4091493}.
The homological vector field $D$ on $ \Cc{\aV}[1]$ is necessarily of
 the form $D=\sum_{n=1}^\infty D_n$,
where $D_n$ is a derivation 
on $C^\infty( \Cc{\aV}[1])$ of weight $(n-1)$:
\begin{equation}
\label{eq:D}
 D_n: \sections{\hat{S}^{\bullet}({\Cc{\aV}^\vee[-1]} )}
\to\sections{\hat{S}^{\bullet+n-1}( {\Cc{\aV}^\vee[-1]})} .
\end{equation}


Under the local  coordinate chart \eqref{eq:localchart2},
 the derivation $D_n$ can be written as
 \begin{equation}
\label{eq:D1}
        D_n = \pm \frac{1}{(n-1)!}\rho^a_{i_1\cdots i_{n-1}}(x)
\eta^{i_1} \cdots \eta^{i_{n-1}} \frac{\p}{\p x^a}
 \pm \frac{1}{n!}C^j_{i_1\cdots i_n}(x) \eta^{i_1}  \cdots \eta^{i_{n}}
\frac{\p}{\p\eta^{j}}.
    \end{equation}

According to Khudaverdian-Voronov \cite{arxiv.1808.10049}, the homological  vector field $D \in \XX (\Cc{\aV}[1])$
and the degree $(+1)$ Poisson tensor $ \Pi\in \sections{ \hat{S}(T_{\Cc{\aV}^\vee[-1]})}$ are related as follows
\begin{equation}
\label{eq:DPi}
\Phi^* D=\Pi,
\end{equation}
where 
\begin{equation}
\label{eq:phi}
\Phi: T^\vee_{{\Cc{\aV}^\vee[-1]}} \xto{\simeq}
 T^\vee_{\Cc{\aV}[1]}
\end{equation}
 is the canonical isomorphism  \cite{MR1262213},  $D \in \XX (\Cc{\aV}[1])$
is identified with its corresponding
 linear function on $T^\vee_{\Cc{\aV}[1]}$,
while $\Pi$ is considered as a  formal polynomial
 on $T^\vee_{{\Cc{\aV}^\vee[-1]}}$.
Following  Shemyakova \cite[Equation (48)]{MR4235776} (see also \cite{MR3894641}),
we denote by
\begin{equation}
\label{eq:Fourier}
\cF: \sections{\Ba{\Cc{\aV}[1]}}\xto{\simeq} \sections{ \Ba{\Cc{\aV}^\vee[-1]}}
\end{equation}
the fiberwise $\hbar$-Fourier transform \footnote{Fiberwise Fourier transform
was  discovered by Voronov-Zorich~\cite{MR947602} in
their study  of supermanifold integration theory. It can
be considered as the quantum counterpart of the isomorphism $\Phi$
 \eqref{eq:phi}. }:
\begin{equation}\label{eq:Fourier1}
    \cF [f(x, \eta)\Cc{D} x^{\,\half} \Cc{D} \eta^{\,\half}] =
 \left(\int_{\Cc{\aV}_x}e^{\ih <\xi, \eta>} f(x,\eta )\,\Cc{D} \eta\right)
\Cc{D}x^{\half} \Cc{D}\xi^{\half}\,.
\end{equation}
where  $(x, \eta)\in  \Cc{\aV}[1]$ and $(x, \xi)\in \Cc{\aV}^\vee[-1]$
(we emphasize the base point $x\in \cM$).
Now set
\begin{equation}
\label{eq:BVinfinity}
\Delta:= \cF \circ \Cc{L}_{\hbar D} \circ \cF^{-1} .
\end{equation}
Since the fiberwise $\hbar$-Fourier transform $\cF $ preserves the  scalar
products and $\Cc{L}_{\hbar D}$ is   self-adjoint, it follows that
$\Delta$ is self-adjoint. Moreover, since $D$ is
a homological vector field, it follows that $\Delta^2=0$.
Finally, it is simple to check, by using local coordinates,
and Eqs. \eqref{eq:pi} and \eqref{eq:D1},  that
$\Delta$ is indeed a quantization of the  linear
$(-1)$-shifted derived Poisson manifold  $\Cc{\aV}^\vee[-1]$.
\end{proof}

\subsection{Quantization of  the
derived intersection of coisotropic submanifolds 
                  of a Lie-Poisson manifold}

Next, as a special case,
we consider the  case when $\Cc{\aV}$ is an ordinary
Lie algebroid $A$ over the base manifold $M$.
 In this situation, $A^\vee [-1]$ is a   $(-1)$-shifted  Poisson
 manifold,
where the degree $(+1)$-Poisson bracket on
$C^\infty (A^\vee [-1])\cong \sections{\Lambda^{-\bullet} A}$ is the
Schouten bracket. 
For the simplicity of notation, assume that the
rank of the vector bundle $A$ is $n$.
By $D: \sections{\Lambda^\bullet A^\vee}\to \sections{\Lambda^{\bullet+1} A^\vee}$,
we denote the  Chevalley--Eilenberg differential of the Lie algebroid.
Thus $(A[1], D)$ is a dg manifold.
Assume, for the sake of simplicity, that
%
$M$ is an orientable manifold and the  vector bundle $A\to M$ is orientable
 as well. Then both  vector bundles $A\oplus T_M$
and $A\oplus T^\vee_M$ are orientable.
We have the following  standard identifications \cite{MR2275685, MR4312284}. 

\begin{equation}
\label{eq:DEN1}
\sections{\Ba{A[1]}}   \cong \sections{  \Lambda A^\vee \otimes \big( \Lambda^{\text{top}} A \otimes
\Lambda^{\text{top}}T^\vee_M\big)^\half }
\end{equation}
while
\begin{equation}
\sections{\Ba{A^\vee[-1]}}   \cong 
 \sections{ \Lambda A \otimes \big( \Lambda^{\text{top}} A^\vee \otimes
\Lambda^{\text{top}}T^\vee_M\big)^\half }.
\label{eq:DEN2}
\end{equation}
Here, by abuse of notation,
$\Lambda^{\text{top}} A \otimes
\Lambda^{\text{top}}T^\vee_M$ and 
 $\Lambda^{\text{top}} A^\vee \otimes
\Lambda^{\text{top}}T^\vee_M$ are identified
with $| \Lambda^{\text{top}} A \otimes
\Lambda^{\text{top}}T^\vee_M|$ and $| \Lambda^{\text{top}} A^\vee \otimes
\Lambda^{\text{top}}T^\vee_M|$, respectively,  and therefore 
their square root make sense.
On the other hand, it is standard \cite[Section 5]{AX03} \cite[Section 6.3]{MR4312284} that there is a canonical isomorphism
\begin{equation}
\Phi:\sections{ \Lambda^k A^\vee \otimes \big( \Lambda^{\text{top}} A \otimes
\Lambda^{\text{top}}T^\vee_M\big)^\half }
\xto{\simeq}
\sections{ \Lambda^{ \text{top}-k} A \otimes \big( \Lambda^{\text{top}} A^\vee \otimes
\Lambda^{\text{top}}T^\vee_M\big)^\half }.
\end{equation}
It is simple to check that, under the identifications \eqref{eq:DEN1}-\eqref{eq:DEN2},
the fiberwise $\hbar$-Fourier transform
\eqref{eq:Fourier} and $\Phi$ are related as follows
\begin{equation}
\cF|_{\tiny {\sections{\Lambda^k A^\vee \otimes \big( \Lambda^{\text{top}} A \otimes
\Lambda^{\text{top}}T^\vee_M\big)^\half }}}
=(\ih)^{n-k}\Phi,
\end{equation}
where $n$ is the rank of $A$.
Thus,  it follows that, under the identifications \eqref{eq:DEN1}-\eqref{eq:DEN2},
$$\cF \circ \Cc{L}_{\hbar D} \circ \cF^{-1}=
\hbar^2 \Phi \circ \Cc{L}_{ D} \circ \Phi^{-1} .$$
It is  known \cite{MR4312284, MR1906481} that, under the identification \eqref{eq:DEN1},
 as a degree $(+1)$ operator on $\sections{\Ba{A[1]}}$,
$\Cc{L}_{ D}$
coincides with the
 Chevalley--Eilenberg differential of the Lie algebroid $A$
with values in the Evens-Lu-Weinstein module  \cite{MR1726784}
$\big( \Lambda^{\text{top}} A \otimes \Lambda^{\text{top}}T^\vee_M\big)^\half $:
$$d_{\text{CE}}^{\text{ELW}}:
\sections{  \Lambda^\bullet  A^\vee \otimes \big( \Lambda^{\text{top}} A \otimes
\Lambda^{\text{top}}T^\vee_M\big)^\half }
\to \sections{  \Lambda^{\bullet+1} A^\vee \otimes \big( \Lambda^{\text{top}} A \otimes
\Lambda^{\text{top}}T^\vee_M\big)^\half }.
$$
As a consequence,  we have the following

\begin{proposition}
\label{pro:SFO}
For any  Lie algebroid $A$ over  $M$, assuming
that $M$ is an orientable manifold and the  vector bundle $A\to M$ is
 orientable as well, the $(-1)$-shifted  Poisson manifold
$A^\vee [-1]$ admits a canonical quantization
\begin{equation}
\Delta=\hbar^2 \Phi \circ d_{\text{CE}}^{\text{ELW}} \circ \Phi^{-1}
: \sections{\Lambda^{-\bullet } A \otimes \big( \Lambda^{\text{top}} A^\vee \otimes
\Lambda^{\text{top}}T^\vee_M\big)^\half }
\to \sections{ \Lambda^{ -(\bullet+1)} A \otimes \big( \Lambda^{\text{top}} A^\vee \otimes
\Lambda^{\text{top}}T^\vee_M\big)^\half }.
\end{equation} 
\end{proposition}

\begin{remark}
Proposition \ref{pro:SFO}, in different forms, have
been known in the literature. When the Lie algebroid
$A$ is  the cotangent Lie algebroid of a Poisson manifold,
this is essentially the Koszul-Brylinski operator
\cite{MR837203, MR950556}. When choosing 
a flat $A$-connection on $\Lambda^{\text{top}} A$, $\Delta$
can be identified with an BV-operator on
$\sections{\Lambda^{-\bullet }A } $ as shown in \cite{MR1675117}.
See also \cite[Section 9.4]{MR4300181} for the   approach
from the viewpoint of factorization  algebras.
\end{remark} 

Finally, we consider the $(-1)$-shifted derived Poisson manifold 
 $(A^\vee [-1], \iota_s, \Pi_2)$ as in Example \ref{example-three},
which arises as the derived  intersection of two coisotropic submanifolds
of the Lie-Poisson manifold $A^\vee$:
 the graph of $s$ and the zero section.

Let $Q= \iota_s$. It is simple to see that, 
under the identification \eqref{eq:DEN2}, the Lie derivative
$$\Cc{L}_Q: \sections{\Ba{A^\vee[-1]}} \to \sections{\Ba{A^\vee[-1]}}$$
becomes  the contraction operator
$$ \iota_s:  \sections{ \Lambda A \otimes \big( \Lambda^{\text{top}} A^\vee
 \otimes \Lambda^{\text{top}}T^\vee_M\big)^\half }
\to
\sections{ \Lambda A \otimes \big( \Lambda^{\text{top}} A^\vee \otimes
\Lambda^{\text{top}}T^\vee_M\big)^\half }
$$
Thus it follows that
 $\Phi^{-1} \circ \iota_s \circ \Phi$,
 the  conjugation by $\Phi$,
 coincides with
the multiplication operator:
$$m_s:  
\sections{  \Lambda^\bullet  A^\vee \otimes \big( \Lambda^{\text{top}} A \otimes
\Lambda^{\text{top}}T^\vee_M\big)^\half }
\to \sections{  \Lambda^{\bullet+1} A^\vee \otimes \big( \Lambda^{\text{top}} A \otimes
\Lambda^{\text{top}}T^\vee_M\big)^\half } .
$$
Since $d_{\text{CE}}s=0$, it follows that 
$$[ d_{\text{CE}}^{\text{ELW}}, \ m_s]=0 .$$
Therefore
$$[ \Phi \circ d_{\text{CE}}^{\text{ELW}} \circ \Phi^{-1} , \ \iota_s]=0 .$$
Hence $\hbar \iota_s+ \hbar^2 \Phi \circ d_{\text{CE}}^{\text{ELW}} \circ \Phi^{-1}$ is  indeed a  $\BV_\infty$-operator quantizing  the
$(-1)$-shifted derived Poisson manifold
 $(A^\vee [-1], \iota_s, \Pi_2)$.

\begin{theorem}
\label{thm:main3}
Let $A$ be a Lie algebroid over $M$, and $s\in \Gamma (A^\vee)$
a Lie algebroid $1$-cocycle. 
 Assume that
 $M$ is an orientable manifold and the  vector bundle $A\to M$ is
 orientable as well.
 Then the $(-1)$-shifted derived Poisson manifold
$(A^\vee [-1], \iota_s, \Pi_2)$ as in Example \ref{example-three}
admits a  canonical quantization
\begin{equation}
\Delta=\hbar \iota_s+ \hbar^2 \Phi \circ d_{\text{CE}}^{\text{ELW}} \circ \Phi^{-1}
: \sections{\Lambda^{-\bullet } A \otimes \big( \Lambda^{\text{top}} A^\vee \otimes
\Lambda^{\text{top}}T^\vee_M\big)^\half }
\to \sections{ \Lambda^{ -(\bullet+1)} A \otimes \big( \Lambda^{\text{top}} A^\vee \otimes
\Lambda^{\text{top}}T^\vee_M\big)^\half }.
\end{equation}
\end{theorem}

In particular, when $A$ is $T_M$ and $s=df\in \Omega^1 (M)$ an exact
one-form, where $f\in C^\infty (M)$, then
$\big( \Lambda^{\text{top}} A \otimes \Lambda^{\text{top}}T^\vee_M\big)^\half $ is canonically isomorphic to the trivial line bundle $M\times \RR$
and the operator
$d_{\text{CE}}^{\text{ELW}}$ reduces to the ordinary de Rham differential.
As an immediate consequence, we see that
\begin{equation}
\hbar \iota_{df}+ \hbar^2 \Phi \circ d_{DR} \circ \Phi^{-1}
: \sections{\Lambda^{-\bullet } T_M \otimes \big( \Lambda^{\text{top}} 
T^\vee_M\big) }
\to 
\sections{\Lambda^{-(\bullet+1) } T_M \otimes \big( \Lambda^{\text{top}} 
T^\vee_M\big) }
\end{equation}
is  a 
$\BV_\infty$-operator quantizing  the
$(-1)$-shifted derived symplectic manifold
 $(T_M^\vee [-1], \iota_{df}, \omega_{\text{can}})$.
Here
$$\Phi: \Omega^k (M)\to \sections{
\Lambda^{ \text{top}-k} T_M \otimes \big( \Lambda^{\text{top}}T^\vee_M\big)
}, \quad \quad \forall k=0, 1, \cdots$$
 is the canonical isomorphism.

\section*{Acknowledgments}

We would like to thank Ruggero Bandiera, Alberto  Cattaneo,
Owen Gwilliam,
Camille Laurent-Gengoux,
Jon Pridham,
Ekaterina Shemyakova, Mathieu Sti\'enon and   Theodore Voronov
 for fruitful discussions and useful comments.
 We are grateful to the anonymous referees for many insightful comments
and suggestions which led to significant improvements in exposition.

Ping Xu is grateful to the  Fields Institute
for hosting the workshop on 
``Supergeometry and Bracket Structures
in Mathematics and Physics" in March 2022, where he  benefited
a lot from  the discussions with the participants.

\section*{Statements and Declarations}

{\bf Competing interests:}
This work was partially supported by 
the NSF grant  DMS-2001599.

{\bf Data availability:}
 Data sharing not applicable to this work since
 no datasets were generated or analyzed in the paper.
\bibliography{references}

\Addresses

\end{document}